\documentclass[a4paper]{amsart}
\usepackage[utf8]{inputenc}
\usepackage[T1]{fontenc}
\usepackage{lmodern, enumerate}
\usepackage{amssymb,amsxtra}
\usepackage{amscd}
\usepackage[all]{xy}
\usepackage{xcolor}
\usepackage{nicefrac,mathtools}
\usepackage{microtype}
\usepackage[pdftitle={Mansfield's Imprimitivity and twisted Landstad Duality},
 pdfauthor={Alcides Buss, Siegfried Echterhoff},
 pdfsubject={Mathematics}
]{hyperref}
\usepackage[lite]{amsrefs}
\newcommand*{\MRref}[2]{ \href{http://www.ams.org/mathscinet-getitem?mr=#1}{MR #1}}
\newcommand*{\arxiv}[1]{\href{http://www.arxiv.org/abs/#1}{arXiv: #1}}

\numberwithin{equation}{section}
\theoremstyle{plain}
\newtheorem{theorem}[equation]{Theorem}
\newtheorem{lemma}[equation]{Lemma}
\newtheorem{proposition}[equation]{Proposition}
\newtheorem{corollary}[equation]{Corollary}
\theoremstyle{definition}
\newtheorem{definition}[equation]{Definition}
\newtheorem{notations}[equation]{Notations}
\theoremstyle{remark}
\newtheorem{remark}[equation]{Remark}
\newtheorem{example}[equation]{Example}

\DeclareMathOperator{\Aut}{Aut}

\DeclareMathOperator{\Rep}{Rep}
\DeclareMathOperator{\Ind}{Ind}

\DeclareMathOperator{\supp}{\mathrm{supp}}
\DeclareMathOperator{\id}{\mathrm{id}}

\DeclareMathOperator{\Inf}{Inf}
\DeclareMathOperator{\im}{Im}

\newcommand*{\nb}{\nobreakdash}
\newcommand*{\Star}{\(^*\)\nobreakdash-}
\newcommand*{\dd}{\,\mathrm{d}}

\newcommand*{\Lb}{\mathcal L}
\newcommand*{\K}{\mathcal K}

\newcommand{\cZ}{\mathcal{Z}}
\newcommand*{\Cst}{C^*}
\newcommand*{\cont}{C}
\newcommand*{\contz}{\cont_0}
\newcommand*{\contc}{\cont_c}
\newcommand*{\contb}{\cont_b}

\newcommand*{\M}{\mathcal M}

\newcommand*{\Id}{\textup{id}}
\newcommand*{\Ad}{\textup{Ad}}
\newcommand*{\U}{\mathcal U}
\newcommand*{\E}{\mathcal E}
\newcommand*{\EE}{\mathbb E}
\newcommand*{\X}{\mathcal X}%

\newcommand*{\defeq}{\mathrel{\vcentcolon=}}
\newcommand*{\congto}{\xrightarrow\sim}

\newcommand*{\braket}[2]{\langle#1\!\mid\!#2\rangle}

\newcommand*{\bbraket}[2]{\mathopen{\langle\!\langle}#1\!\mid\!#2\mathclose{\rangle\!\rangle}}

\newcommand*{\sbe}{\subseteq} 
\newcommand*{\F}{\mathcal F}
\newcommand*{\D}{\mathcal D}
\newcommand*{\cstar}{\texorpdfstring{$C^*$\nobreakdash-\hspace{0pt}}{*-}}
\newcommand*{\into}{\hookrightarrow}
\newcommand*{\onto}{\twoheadrightarrow}
\newcommand*{\red}{r}

\newcommand*{\un}{u}
\newcommand*{\pn}{\mathrm{\mu}}

\newcommand*{\dual}[1]{\widehat{#1}}
\newcommand*{\dualG}{\widehat{G}}
\newcommand*{\dualGN}{\widehat{G/N}}

\newcommand*{\dualalpha}{\widehat{\alpha}}
\newcommand*{\twu}{\omega}
\newcommand*{\twh}{\varsigma}
\newcommand{\Zt}{\mathcal{Z}} 

\newcommand*{\pt}[1]{\dot{#1}}
\newcommand*{\tw}{\upsilon}
\newcommand*{\st}{\mathrm{st}}
\newcommand*{\bs}{\backslash}

\newcommand{\ie}{\emph{i.e.}}
\newcommand{\eg}{\emph{e.g.}}
\newcommand{\cf}{\emph{cf.}}

\begin{document}
\title[Mansfield's Imprimitivity and twisted Landstad Duality]{Weakly proper group actions, Mansfield's imprimitivity and twisted Landstad duality}

\author{Alcides Buss}
\email{alcides.buss@ufsc.br}
\address{Departamento de Matem\'atica\\
 Universidade Federal de Santa Catarina\\
 88.040-900 Florian\'opolis-SC\\
 Brazil}

\author{Siegfried Echterhoff}
\email{echters@uni-muenster.de}
\address{Mathematisches Institut\\
 Westf\"alische Wilhelms-Universit\"at M\"un\-ster\\
 Einsteinstr.\ 62\\
 48149 M\"unster\\
 Germany}

\begin{abstract}
Using the theory of weakly proper actions of locally compact groups recently developed by the authors, we give a unified proof of both reduced and maximal versions of Mansfield's Imprimitivity Theorem and obtain a general version of Landstad's Duality Theorem for twisted group coactions. As one application, we obtain the stabilization trick for arbitrary twisted coactions, showing that every twisted coaction is Morita equivalent to an inflated coaction.
\end{abstract}



\thanks{Supported by Deutsche Forschungsgemeinschaft  (SFB 878, Groups, Geometry \& Actions) and by CNPq (Ciências sem Fronteira) -- Brazil.}

\maketitle

\section{Introduction}
\label{sec:introduction}

The main goal of this paper is to show that the theory of weakly proper actions of locally compact groups developed by the authors in \cite{Buss-Echterhoff:Exotic_GFPA,Buss-Echterhoff:Imprimitivity}
can be used to give unified proofs  and/or generalizations of some of the central results  about (twisted) coactions of groups.
More specifically, we want to explore Mansfield's Imprimitivity Theorem (and its generalizations) as well as Landstad Duality for twisted coactions of groups from the point of view of the theory of weakly proper actions and their generalized fixed-point algebras.

In \cite{Mansfield:Induced} Mansfield proved his main result, today called Mansfield's Imprimitivity Theorem, which says that for a (reduced) coaction $\delta\colon B\to \M(B\otimes C^*_\red(G))$ of
a locally compact group $G$ on a \cstar{}algebra $B$ and an amenable normal closed subgroup $N\sbe G$, the crossed product $B\rtimes_{\delta|}\dual{G/N}$ by the restricted coaction $\delta|\colon B\to \M(B\otimes C^*_\red(G/N))$ of $G/N$ is Morita equivalent to $B\rtimes_\delta\dualG\rtimes_{\dual\delta}N$, the crossed product by the dual $N$-action $\dual\delta$.
The bimodule implementing this equivalence is obtained as a certain completion of a special dense \Star{}subalgebra
$\D\sbe B\rtimes_\delta\dualG$, often called {\em Mansfield subalgebra}. Over time, several authors -- see \cite{anHuef-Raeburn:Mansfield,Kaliszewski-Quigg:Mansfield,Kaliszewski-Quigg:Imprimitivity} -- generalized Mansfield's theorem in different directions by allowing non-amenable and even non-normal closed subgroups of $G$ in combination with
different classes of coactions including  {\em full} normal or maximal coactions of $G$ (the word "full" means that we consider
  coactions of the full group \cstar{}algebra $C^*(G)$). We should emphasize that the theory of "full normal" coactions is equivalent to the theory of coactions
  by the reduced group algebra $C_r^*(G)$ (see \cite{Quigg:FullAndReducedCoactions}).

The  version of Mansfield's theorem for normal coactions can be obtained from the theory of Rieffel proper actions (\cite{Rieffel:Proper,Rieffel:Integrable_proper}) by proving that the dual action of $N$ on $B\rtimes_\delta \dualG$
is proper (in Rieffel's sense) with respect to Mansfield's subalgebra $\D\sbe B\rtimes_\delta\dualG$. Indeed, this fact has been first observed  in Mansfield's original paper
(see \cite[\S 7]{Mansfield:Induced}) and  it has been used to obtain generalizations of Mansfield's Imprimitivity Theorem to non-normal and/or non-amenable subgroups
in  \cite{Kaliszewski-Quigg:Imprimitivity, Huef-Kaliszewski-Raeburn-Williams:Fixed,anHuef-Raeburn:Mansfield}. On the other hand, the maximal version of Mansfield's theorem (obtained in \cite{Kaliszewski-Quigg:Mansfield}) is proved in an indirect way by
analysing relations between several imprimitivity bimodules (such as Green's imprimitivity bimodule and Katayama's bimodule).

One of our goals in this paper is to show that both, the maximal and normal versions of Mansfield's Imprimitivity Theorem can be obtained
by considering full or reduced generalized fixed-point algebras for appropriate weakly proper actions.
While the reduced generalized fixed-point algebras have been introduced by Rieffel in the 1980's (\cite{Rieffel:Proper}), the theory of full fixed-point algebras has been introduced only recently in the quite general situation of weakly proper $G$-algebras  by the authors in
\cite{Buss-Echterhoff:Exotic_GFPA}.
Recall that a $G$-action $\alpha$ on a \cstar{}algebra $A$ is called weakly proper if there is a proper $G$-space
$X$ and a $G$-equivariant nondegenerate \Star{}homomorphism $\contz(X)\to \M(A)$.
We then call $A$  a weakly proper $X\rtimes G$-algebra (or just a weak $X\rtimes G$-algebra).
 For such algebras we  constructed in \cite{Buss-Echterhoff:Exotic_GFPA} a Hilbert module
 $\F_\pn(A)$ over the $\pn$-crossed product $A\rtimes_{\alpha,\pn}G$ for any given crossed-product norm $\|\cdot\|_\pn$ on $\contc(G,A)$ which lies
 between the reduced  crossed-product norm $\|\cdot\|_\red$ and the maximal crossed-product norm $\|\cdot\|_\un$.
The algebra of compact operators $A^G_\pn=\K(\F_\pn(A))$ is a completion of the {\em generalized fixed-point algebra with compact supports:}
\begin{equation}\label{eq:GFPAwithCompSupp}
A^G_c=\contc(G\bs X)\cdot \{m\in \M(A)^G: m\cdot \contc(X), \contc(X)\cdot m\sbe A_c\}\cdot \contc(G\bs X),
\end{equation}
where $A_c=\contc(X)\cdot A\cdot \contc(X)$ and $\M(A)^G$ denotes the
algebra of $G$-fixed points in the multiplier algebra $\M(A)$.
If the action of $G$ on $X$ is free and proper,  $\F_\pn(A)$ implements a Morita equivalence $A^G_\pn\sim A\rtimes_{\alpha,\pn}G$.

Given a $G$-coaction $(B,\delta)$, the crossed product $B\rtimes_\delta \dualG$ may be viewed as a weak $G\rtimes G$-algebra in a canonical
way by taking the dual $G$-action $\dual\delta$ and the canonical homomorphism $j_B\colon\contz(G)\to \M(B\rtimes_\delta\dualG)$,
where $X=G$ is endowed with the  right translation action of $G$.
In particular, if $H$ is a closed subgroup of $G$, we may restrict the $G$-action to $H$ and view $B\rtimes_\delta\dualG$ as a weak
$G\rtimes H$-algebra. Therefore, by the general theory of weakly proper actions explained above, we get a Hilbert bimodule
$\F^H_\pn(B\rtimes_\delta \dualG)$ implementing a Morita equivalence
$(B\rtimes_\delta\dualG)^H_\pn\sim B\rtimes_\delta\dualG\rtimes_{\dual\delta,\pn}H$ for any crossed-product norm
$\|\cdot\|_\pn$ on $\contc(H,B\rtimes_\delta\dualG)$.
The only remaining point to get Mansfield's theorem is to suitably identify the fixed point algebra
$(B\rtimes_\delta\dualG)^H_\pn$ with a sort of "crossed product" $B\rtimes_{\delta|,\pn}\dual{G/H}$
by the homogeneous space $G/H$ whenever this is defined. In fact, we prove that if $N$ is a
normal closed subgroup of $G$ and $\pn=\un$ or $\pn=\red$ denotes either the maximal or
reduced crossed-product norm (for both groups $G$ and $N$), then $(B\rtimes_\delta\dualG)^N_\pn$
is indeed isomorphic to the  crossed product $B_\pn\rtimes_{\delta_\pn|}\dual{G/N}$ by the restricted coaction,
where $(B_\pn,\delta_\pn)$ denotes either the maximalization (for $\pn=\un$) or the
normalization (for $\pn=\red$) of $(B,\delta)$. For non-normal subgroups $H\sbe G$, it follows almost by definition
that $(B\rtimes_\delta\dualG)^H_\red$
identifies with the crossed product $B_\red\rtimes_{\delta,\red}\dual{G/H}$ as defined in
\cite{Echterhoff-Kaliszewski-Raeburn:Crossed_products_dual}. This has been observed before in
\cite{anHuef-Raeburn:Mansfield}*{Theorem~3.1} and \cite{Huef-Kaliszewski-Raeburn-Williams:Naturality_Rieffel}*{Proposition~5.2}.
Our results indicate that it would be useful to {\em define} the {\em full crossed product}
$B\rtimes_{\delta,\un}\dual{G/H}$ of $B$ by the restriction of a $G$-coaction $\delta$ to the homogeneous space $G/H$
 as the maximal fixed-point algebra  $(B\rtimes_\delta\dualG)^H_\un$. Our
Morita equivalence $(B\rtimes_\delta\dualG)^H_\un\sim  (B\rtimes_\delta\dualG)\rtimes_{\widehat{\delta}, \un}H$  then automatically provides a
 full version of Mansfield's Imprimitivity Theorem for crossed products by   $G/H$.

We should stress that our results are completely independent from Mansfield's original ideas of constructing certain
dense subalgebras $\D_H$ and $\D$ of $B\rtimes_{\delta, r}\dual{G/H}$ and $B\rtimes_\delta\dualG$, respectively.
Nevertheless, we show that our results are compatible with Mansfield's constructions by showing that
the algebra $\D_H$ of Mansfield lies inductive limit dense in the fixed-point algebra $(B\rtimes_{\delta}\dualG)_c^H$, and hence also dense in any fixed-point algebra $(B\rtimes_\delta\dualG)^H_\pn$.
In particular, this implies that one also obtains the full fixed-point algebra $(B\rtimes_{\delta}\dualG)_\un^H$ as
a certain completion of $\D_H$.

Mansfield's theorem motivated Phillips and Raeburn to introduce twisted coactions of groups in \cite{Phillips.Raeburn:Twisted}.
If $N\sbe G$ is a closed normal subgroup of $G$, a twist over $G/N$ for a (full) coaction $(B,\delta)$ of $G$ is a unitary corepresentation $\twu\in \M(B\otimes C^*(G/N))$ of $G/N$
such that the restriction $\delta|$ of $\delta$ to $G/N$ is implemented by conjugation with $\twu$ and such that $\delta$ coacts trivially on the first leg of $\twu$.
For each such twisted coaction, one can form the twisted crossed product $B\rtimes_{\delta,\twu}\dualG$ which is the quotient of the untwisted crossed product
$B\rtimes_\delta\dualG$ by a certain
 (twisting) ideal.
Another main result of this paper will be a Landstad Duality Theorem for twisted coactions of groups which will, in particular,
 provide us with the notion of {\em maximalizations} of twisted coactions:
we prove that for a weak $G\rtimes N$-algebra $A$,
there is a twisted coaction $(\delta^N_\pn,\twu^N_\pn)$ on $A^N_\pn$ -- for suitable crossed-product norms $\pn$ on $C_c(N,A)$ -- and a natural isomorphism
 $A^N_\pn\rtimes_{\delta_\pn,\twu_\pn}\dualG\cong A$ of weak $G\rtimes N$-algebras.
Conversely, if we start with any given twisted coaction $(\delta, \twu)$, the twisted crossed product
$A:=B\rtimes_{\delta,\twu}\dualG$ carries a canonical structure as weak $G\rtimes N$-algebra, and
for suitable crossed-product norms $\|\cdot\|_\pn$ on $C_c(N,B\rtimes_{\delta,\twu}\dualG)$ we obtain
a twisted $(G,G/N)$-coaction $(\delta^N_\pn,\twu^N_\pn)$ on  $B^N_\pn:=(B\rtimes_{\delta,\twu}\dualG)_\pn^N$ such that
$\F_\pn^N(B\rtimes_{\delta,\twu}\dualG)$ implements a  Morita equivalence
between $B_\pn^N$ and $B\rtimes_{\delta,\twu}\dualG\rtimes_{{\dual\delta,}\pn} N$ which provides a version
of Katayama duality for such twisted cosystems.  We then show:

\begin{enumerate}

\item Given an arbitrary twisted coaction $(B,\delta,\twu)$ there exists a unique norm $\|\cdot\|_\pn$ on $C_c(N, B\rtimes_{\delta,\twu}\dualG)$
such that $(B,\delta,\twu)\cong (B_\pn^N,\delta_\pn^N, \twu_\pn)$.  In particular, $(B,\delta,\twu)$ satisfies the above version of  Katayama duality
for $\|\cdot\|_\pn$. Moreover, we show that
$(B,\delta,\twu)$ is Morita equivalent to the trivially twisted inflated bidual coaction $(\Inf\dual{\dual\delta}_\pn, 1)$ on $B\rtimes_{\delta,\twu}\dualG\rtimes_{{\dual\delta,}\pn} N$.
This gives the \emph{stabilization trick for arbitrary twisted coactions} extending the main result of  \cite{Echterhoff-Raeburn:The_stabilisation_trick} where the stabilization trick was shown for amenable $N$.

\item There are canonical epimorphisms $B_u^N\onto B\onto B_\red^N$ which are equivariant  for the twisted coactions
$(\delta^N_\un,\twu^N_\un)$, $(\delta,\twu)$, and $(\delta^N_\red,\twu^N_\red)$, respectively, such that the resulting homomorphisms
$$B_u^N\rtimes_{\delta_u^N,\twu_u}\dualG\onto B\rtimes_{\delta,\twu}\dualG\onto B_\red^N\rtimes_{\delta_\red^N,\twu_\red}\dualG$$
are isomorphisms of weakly proper $G\rtimes N$-algebras.
\end{enumerate}

By the previous discussion we already know that the twisted cosystem $ (B_\un^N,\delta_\un^N, \twu_\un)$ satisfies full Katayama duality
$B_\un^N\sim_M (B_u^N\rtimes_{\delta_u^N,\twu_u}\dualG)\rtimes_uN$ and the twisted cosystem
 $ (B_\red^N,\delta_\red^N, \twu_\red)$ satisfies reduced Katayama duality
$B_\red^N\sim_M (B_\red^N\rtimes_{\delta_\red^N,\twu_\red}\dualG)\rtimes_\red N$. They  therefore
give twisted analogues of a {\em maximalization} and a {\em normalization} of the  coaction
$(\delta,\twu)$, thus extending  similar concepts for ordinary coactions as introduced in \cite{Echterhoff-Kaliszewski-Quigg:Maximal_Coactions}
and \cite{Quigg:FullAndReducedCoactions}  and we obtain complete twisted analogues of the  results on exotic coactions
obtained in \cite{Buss-Echterhoff:Exotic_GFPA}.

The outline of the paper is as follows: after a short preliminary section (\S\ref{sec-prel}) we prove in \S \ref{sec-iterated-fixed}
some useful general results on fixed-point algebras including  a theorem on iterated fixed-point algebras for
normal subgroups: if $A$ is a weakly proper $X\rtimes G$-algebra and $N$ is a closed normal subgroup of $G$, then
$A_u^G\cong (A_u^N)_u^{G/N}$ and similarly for the reduced fixed-point algebra. This result will give the main tool for
proving our versions of Mansfield's theorem in \S \ref{sec-Mansfield}. The results on twisted  Landstad Duality and
maximalizations and normalizations for twisted coactions are given in the final section \S \ref{sec:Twisted-Landstad}.
As one application, we  give a new proof of the decomposition theorem
$B\rtimes_{\delta}\dualG\cong \big(B\rtimes_{\delta|}\dual{G/N}\big)\rtimes_{\tilde\delta,\tilde\twu}\dualG$
of Phillips and Raeburn in \cite{Phillips.Raeburn:Twisted}.

\section{Preliminaries}\label{sec-prel}
To fix notation and for reader's convenience we recall in this section some basic definitions and constructions from \cite{Buss-Echterhoff:Exotic_GFPA, Buss-Echterhoff:Imprimitivity}
about weakly proper actions and their generalized fixed-point algebras as well as some notations on  coactions that will be needed in this paper.

Let $G$ be a locally compact group and let $A$ be a $G$-algebra, that is, a \cstar{}algebra endowed with a (strongly continuous) $G$-action $\alpha\colon G\to \Aut(A)$.
We endow $\contc(G,A)$ with the usual \Star{}algebra structure by
$$f*g(t)\defeq \int_G f(s)\alpha_s(g(s^{-1}t))\dd{s},\quad f^*(t)\defeq \Delta(t)^{-1}\alpha_t(f(t^{-1})^*),$$
where $\Delta$ denotes the modular function of $G$. The full and reduced crossed product, denoted $A\rtimes_\alpha G$ and $A\rtimes_{\alpha,\red}G$, respectively, are also defined in the usual way as completions of $\contc(G,A)$ with respect to the universal and reduced \cstar{}norms $\|\cdot\|_\un$ and $\|\cdot\|_\red$, respectively  -- the latter is defined in terms of the regular representation $\Lambda\colon \contc(G,A)\to \Lb(L^2(G,A))$.
More generally, we call a \emph{crossed-product norm} any \cstar{}norm $\|\cdot\|_\pn$ between the full and reduced norm and write $A\rtimes_{\alpha,\pn}G$ for the corresponding \cstar{}algebra completion, sometimes called an \emph{exotic crossed product}. Certain special exotic crossed product of this type have been constructed in \cite{Kaliszewski-Landstad-Quigg:Exotic}: they are associated to crossed-product norms coming from $G$-invariant ideals in the Fourier-Stieltjes algebra $B(G)$ of $G$.

Let now $X$ be a locally compact Hausdorff $G$-space with left $G$-action $G\times X\to X$, $(t,x)\mapsto t\cdot x$. We usually denote
by $\tau: G\to\Aut(C_0(X))$ the corresponding action given by $\big(\tau_t(f)\big)(x)=f(t^{-1}\cdot x)$.
An important special situation will be the case where $X=G$ is endowed with (right) translation $G$-action $t\cdot g:=gt^{-1}$.

By a \emph{weak $X\rtimes G$-algebra} we mean a \cstar{}algebra $A$ endowed with a $G$-action $\alpha$ and a $G$-equivariant nondegenerate
\Star{}homomorphism $\phi\colon\contz(X)\to \M(A)$.  We often write  $f\cdot a\defeq \phi(f)a$ and $a\cdot f\defeq a\phi(f)$ for  $f\in \contz(X)$, $a\in A$.
We are mainly interested in the situation where $X$ is a proper $G$-space, in which case we say that $A$ is a weakly proper $X\rtimes G$-algebra
(notice that these actions are also proper in Rieffel's sense \cite{Rieffel:Proper} by \cite{Rieffel:Integrable_proper}*{Theorem~5.7}).
 Recall that the $G$-action on $X$ is proper if and only if for every compact subsets $K,L\sbe X$, the set $\{t\in G: t\cdot K\cap L\not=\emptyset\}$ is compact in $G$. In this situation, the space $\F_c(A)\defeq \contc(X)\cdot A$ can be endowed with a canonical structure of a pre-Hilbert module
over $\contc(G,A)\sbe A\rtimes_{\alpha,\pn}G$ (for any crossed-product norm $\|\cdot\|_\pn$) with inner product and right $\contc(G,A)$-action given by
$$\bbraket{\xi}{\eta}_{\contc(G,A)}|_t\defeq \Delta(t)^{-1/2}\xi^*\alpha_t(\eta),\quad \xi*f=\int_G\Delta(t)^{-1/2}\alpha_t(\xi\cdot f(t^{-1}))\dd{t}$$
for all $\xi,\eta\in \F_c(A)$, $f\in \contc(G,A)$ and $t\in G$. The Hilbert $A\rtimes_{\alpha,\pn}G$-module completion of $\F_c(A)$ is denoted by $\F_\pn(A)$ (sometimes also
$\F_\pn^G(A)$ if it is important to keep track of the group $G$).
The \cstar{}algebra of compact operators on $\F_\pn(A)$ can be canonically identified with a completion $A^G_\pn$ of the generalized fixed-point algebra with compact supports $A^G_c$ (see \eqref{eq:GFPAwithCompSupp})
via the (left) $A^G_c$-valued inner product and left action on $\F_c(A)$ given by
$$_{A^G_c}\bbraket{\xi}{\eta}\defeq \EE(\xi\eta^*)=\int_G^\st\alpha_t(\xi\eta^*)\dd{t},\quad a\cdot \xi\defeq a\xi \mbox{ (multiplication in $\M(A)$)},$$
where $\EE(a)\defeq \int_G^\st\alpha_t(a)\dd{t}$ denotes the strict (unconditional) integral (\cite{Exel:Unconditional})
whenever this makes sense (which is the case for elements $a=\xi\eta^*\in A_c\defeq \contc(X)\cdot A\cdot\contc(X)$).
We shall write $\EE^G$ if it is important to keep track of the group in the notation.
Note that in this construction $\F_c(A)$ becomes a (partial) $A_c^G$--$C_c(G,A)$-pre-imprimitivity bimodule in which all
possible pairings are jointly continuous with respect to the respective inductive limit topologies (see
\cite[Definition 2.11]{Buss-Echterhoff:Exotic_GFPA} for the definition of these topologies). This is shown in
 \cite[Lemma 2.12]{Buss-Echterhoff:Exotic_GFPA} and the proof of that lemma also shows that the $\EE:A_c\to A_c^G$ is  inductive limit continuous as well.

If the above construction is applied to the special case where $A=\contz(X)$, we obtain a $C_0(X)\rtimes_\tau G$-Hilbert module
$\F(X)\defeq \F(\contz(X))$ with algebra of compact operators isomorphic to $\contz(G\bs X)$.
Recall that $\contz(X)\rtimes_{\tau}G\cong \contz(X)\rtimes_{\tau,\red}G$, that is, there is only one crossed-product norm $\pn=\un=\red$ on $\contc(G,\contz(X))$ because the action is proper.
This is definitely not the case in general: every exotic crossed product appears as an (exotic) $\pn$-generalized fixed point algebra $A^G_\pn$ (see \cite{Buss-Echterhoff:Imprimitivity}*{Corollary~3.25}).
The relation between $\F(X)$ and $\F_\pn(A)$ is given by the (balanced) tensor product decomposition (obtained in \cite{Buss-Echterhoff:Exotic_GFPA}*{Proposition~2.9}):
$$\F_\pn(A)\cong \F(X)\otimes_{\contz(X)\rtimes_\tau G}(A\rtimes_{\alpha,\pn}G).$$
In particular, if the action on $X$ is free,
 this decomposition implies that $\F_\pn(A)$ is full as a right Hilbert $A\rtimes_{\alpha,\pn}G$-module and hence may be viewed as an imprimitivity bimodule between
$A^G_\pn$ and  $A\rtimes_{\alpha,\pn}G$.
One special class of weakly proper actions where the above theory of generalized fixed point algebras can be successfully applied comes from crossed products by
group coactions. Recall that a (full) coaction of $G$ on a \cstar{}algebra $B$ is a nondegenerate \Star{}homomorphism $\delta\colon B\to \M(B\otimes C^*(G))$ satisfying $(\delta\otimes\id)\circ\delta=(\id\otimes\delta_G)\circ\delta$
and such that $\delta(B)(1\otimes C^*(G))=B\otimes C^*(G)$ ("nondegeneracy" of the coaction), where $\delta_G\colon C^*(G)\to \M(C^*(G)\otimes C^*(G))$, $\delta_G(u_t)=u_t\otimes u_t$ for all $t\in G$, denotes the comultiplication of $C^*(G)$ and $G\ni t\mapsto u_t\in \M(C^*(G))$ denotes the universal representation. Although we have mentioned reduced coactions
 in the introduction, meaning (injective) coactions modelled on $C^*_\red(G)$ in place of $C^*(G)$, we only work with full coactions in the main body of this paper.

Given a coaction $(B,\delta)$, one can assign the crossed product $B\rtimes_\delta\dualG$ which is endowed with a universal covariant representation pair $j_B\colon B\to \M(B\rtimes_\delta\dualG)$ and $j_G\colon \contz(G)\to \M(B\rtimes_\delta\dualG)$ in such a way that elements of the form $j_B(b)j_G(f)$ linearly span a dense subspace of $B\rtimes_\delta\dualG$. By a covariant representation we mean a pair of (nondegenerate) \Star{}homomorphisms
$\pi,\sigma\colon B,\contz(G)\to \M(D)$, for some \cstar{}algebra $D$, satisfying $$(\pi\otimes\id)(\delta(b))=(\sigma\otimes\id)(\omega_G)(\pi(a)\otimes 1)(\sigma\otimes\id)(\omega_G)^*\quad\mbox{for all }b\in B,$$
where $\omega_G\in \M(\contz(G)\otimes C^*(G))$ is the unitary represented by the function $t\mapsto u_t$. The universality of $(j_B,j_G)$ means that any such pair $(\pi,\sigma)$ gives rise to a (unique) nondegenerate \Star{}homomorphism $\pi\rtimes\sigma\colon B\rtimes_\delta\dualG\to \M(D)$ with $(\pi\rtimes\sigma)\circ j_B=\pi$ and $(\pi\rtimes\sigma)\circ j_G=\sigma$. The theory of crossed products by coactions turns out to be "amenable", in the sense that the regular representation of $B\rtimes_\delta\dualG$ into $\M(B\otimes\K(L^2G))$ given by the covariant pair $(\pi,\sigma)=((\id\otimes\lambda)\circ\delta, 1\otimes M)$, where $\lambda$ denotes the left regular representation of $G$ and $M$ the representation of $\contz(G)$ by multiplication operators, is faithful for every coaction. In other words, $B\rtimes_\delta\dualG$ may be identified with the image of the regular representation in $\M(B\otimes\K(L^2G))$. On the other hand, the representation $j_B$
need not be faithful in general (as happens for some dual coactions on full crossed products by actions of non-amenable $G$). A coaction is said to be \emph{normal} if $j_B$ is injective.

The crossed product $B\rtimes_\delta\dualG$ carries a \emph{dual action} $\dual\delta$ of $G$ given on generators by the formula:
$$\dual\delta_t(j_B(b)j_G(f))=j_B(b)j_G(\tau_t(f)),$$
where $\tau$ denotes the right translation action of $G$ on itself: $\tau_t(f)|_s=f(st)$. This action turns $j_G$ into a $G$-equivariant homomorphism and therefore enriches $A=B\rtimes_\delta\dualG$
with the structure of a weakly proper $G\rtimes G$-algebra. The double (full) crossed product $B\rtimes_\delta\dualG\rtimes_{\dual\delta}G$ is, in general, not
isomorphic to $B\otimes\K(L^2G)$, but there is a canonical surjection
$$\Phi\colon B\rtimes_\delta\dualG\rtimes_{\dual\delta}G\onto B\otimes\K(L^2G)$$
which is defined as the integrated form $\Phi\defeq (\pi\rtimes\sigma)\rtimes (1\otimes \rho)$, where $(\pi,\sigma)=((\id\otimes\lambda)\circ\delta, 1\otimes M)$ is the regular covariant representation of $(B,\delta)$ and $\rho$ denotes the right regular representation of $G$ on $L^2G$. The coaction $(B,\delta)$ is called \emph{maximal} if $\Phi$ is an isomorphism. Maximal coactions are exactly those which are Morita equivalent to dual coactions on full crossed products by actions. In general, there is a (unique, up isomorphism) \emph{maximalization} $(B_\un,\delta_\un)$ of $(B,\delta)$ which is a maximal coaction together with an equivariant surjection $B_\un\to B$ inducing an isomorphism $B_\un\rtimes_{\delta_\un}\dualG\congto B\rtimes_\delta\dualG$ (of weak $G\rtimes G$-algebras). Similarly, there is a \emph{normalization} $(B_\red,\delta_\red)$ of $(B,\delta)$, that is, a normal coaction with an equivariant surjection $B\to B_\red$ inducing an isomorphism $B\rtimes_\delta\dualG\cong B_\red\rtimes_{\delta_\red}\dualG$ (of weak $G\rtimes G$-algebras) in such a  way that the canonical surjection $\Phi$ factors through an isomorphism
 $$\Phi_\red\colon B\rtimes_\delta\dualG\rtimes_{\dual\delta,\red}G\congto B_\red\otimes\K(L^2G).$$
More generally, $\Phi$ determines a crossed-product norm $\|\cdot\|_\pn$ on $\contc(G,A)$ for $A=B\rtimes_\delta\dualG$ in such a way that $\Phi$ factors through an isomorphism
$$\Phi_\pn\colon B\rtimes_\delta\dualG\rtimes_{\dual\delta,\pn}G\onto B\otimes\K(L^2G).$$
It follows from \cite{Buss-Echterhoff:Exotic_GFPA}*{Theorem~4.6} that there is a coaction $\delta^G_\pn$ on the $\pn$-generalized fixed-point algebra $A^G_\pn$
and Corollary~4.7 in \cite{Buss-Echterhoff:Exotic_GFPA} says that $(B,\delta)\cong (A^G_\pn,\delta^G_\pn)$. In this situation we say that $(B,\delta)$ is a $\pn$-coaction,
or that it satisfies $\pn$-duality (which is implemented by $\Phi_\pn$). In particular, a coaction is maximal (resp. normal) if and only if $(B,\delta)\cong (A^G_\un,\delta^G_\un)$ (resp. $(B,\delta)\cong (A^G_\red,\delta^G_\red)$).

Summarizing, we may recover every coaction of $G$ as a coaction of the form $(A^G_\pn,\delta^G_\pn)$ for some weak $G\rtimes G$-algebra $A$.
Moreover, for crossed-product norms associated to ideals in $B(G)$ as in \cite{Kaliszewski-Landstad-Quigg:Exotic}, the assignment $A\mapsto (A^G_\pn,\delta^G_\pn)$ is an equivalence between the categories of weak $G\rtimes G$-algebras and $G$-coactions satisfying $\pn$-duality (see \cite{Buss-Echterhoff:Exotic_GFPA}*{Theorem~7.2}).

In this paper we extend these results and describe the category of weak $G\rtimes N$-algebras, for $N$ a closed normal subgroup of $G$, in terms of coactions of $G$ twisted over $G/N$. This will be done in Section~\ref{sec:Twisted-Landstad}, where we review the definition of twisted coactions and derive the relevant results.

\section{Green twisted actions and iterated fixed-point algebras}\label{sec-iterated-fixed}

Let $G$ be a locally compact group and $N\sbe G$ a normal closed subgroup.
In most of this section $\|\cdot\|_\pn$ will denote either the maximal or reduced crossed-product norm.

Let $(B,\beta)$ be a $G$-algebra. Recall that a (Green) \emph{twist} for $\beta$ is a strictly continuous group homomorphism $\tw\colon N\to \U\M(A)$ satisfying
$$\alpha_n(a)=\tw_na\tw_{n^{-1}}\quad\mbox{and}\quad\alpha_t(\tw_n)=\tw_{tnt^{-1}}\quad\forall t\in G, n\in N.$$
In this case we also say that $(\beta,\tw)$ is a (Green) {\em twisted action} of $(G,N)$ on a \cstar{}algebra $B$, or that $(B,\beta,\tw)$ is a $(G,N)$-algebra.
If $\tw$ is the trivial twist, that is, $\tw_n=1$ for all $n$, then $\beta$ is trivial on $N$ and hence factors through a $G/N$-action $\dot\beta$. Conversely,
If $(B,\pt\beta)$ is a $G/N$-algebra, then we may \emph{inflate} $\pt\beta$ to a $G$-action $\Inf\pt\beta$ on $B$ and this is a $(G,N)$-twisted action with respect to the trivial twist.
Hence, we may view $(G,N)$-algebras as generalizations of $G/N$-algebras.
The maximal twisted crossed product $B\rtimes_{\beta,v}(G,N)$
can be constructed as the universal completion of the convolution algebra $C_c(G, B,v)$ consisting of all
continuous functions $f:G\to B$ with compact supports modulo $N$ which satisfy the relation
$f(ns)=f(s)v_{n^{-1}}$ for all $s\in G, n\in N$. Convolution and involution on  $C_c(G,B,v)$ are defined by
$$f*g(s)=\int_{G/N} f(t)\beta_t(g(t^{-1}s))\,\dd_{G/N} tN\quad\text{and}\quad f^*(s)=\Delta_G(s^{-1})\beta_s(f(s^{-1})^*.$$
 We always choose Haar measures on $G, N$, and $G/N$ in such a way that the formula
\begin{equation}\label{iterated-int}
\int_G \varphi(s)\,\dd_G s=\int_{G/N}\left(\int_N \varphi(sn)\, \dd_N n\right)\dd_{G/N} sN
\end{equation}
holds for all $\varphi\in C_c(G)$.
The nondegenerate \Star{}representations of $B\rtimes_{\beta,v}G$ are in one-to-one correspondence with the
covariant representations $(\pi,U)$ of $(B,\beta)$ which preserve the twist $v$ in the sense that
$\pi(v_n)=U_n$ for all $n\in N$. Any such covariant representation integrates to a \Star{}representation $\pi\rtimes U$ of $C_c(G,B,v)$
by putting
$$\pi\rtimes U(f)=\int_{G/N} \pi(f(s)) U_s\, d_{G/N} sN.$$
The universal norm $\|\cdot\|_u$ on $C_c(G,B,v)$ is then given as $\|f\|_u=\sup_{(\pi,U)}\|\pi\rtimes U(f)\|$ where $(\pi,U)$ runs through
all twisted covariant representations of $(B, \beta, \tw)$. Alternatively, $B\rtimes_{\beta,\tw}(G,N)$ can be obtained as the quotient
of the untwisted crossed product by the ideal
$$I_v:=\cap\{\ker\pi\rtimes U: (\pi,U)\;\text{is a twisted covariant representation of $(B,\beta,v)$}\}.$$
Note that we  have a canonical isomorphism $C_c(G,B,1_N)\cong C_c(G/N,B)$ which induces an isomorphism
$B\rtimes_{\Inf\dot\beta, 1_N}(G,N)\cong B\rtimes_{\dot\beta}G/N$, if $\dot\beta$ is an action of $G/N$ on $B$.
Twisted actions of this kind have been introduced by Phil Green in \cite{Green:Local_twisted} and we refer to his
paper for more details.

If  $(A,\alpha, \tw^\alpha)$ and $(B,\beta, \tw^\beta)$ are two twisted $(G,N)$-algebras then a $(G,N)$-equi\-va\-riant
$A$--$B$-correspondence $(\E,\gamma)$ is a $G$-correspondence $(\E,\gamma)$
between $(A,\alpha)$ and $(B,\beta)$ which preserves the twists in the sense that
$$ \gamma_n(\xi)=\tw^\alpha_n\cdot \xi \cdot \tw^\beta_{n^{-1}} \quad\forall \xi\in \E, n\in N.$$
 If, in addition, $\E$ is an imprimitivity $A$--$B$-bimodule, we say that the twisted actions $(A,\alpha,\tw^\alpha)$ and $(B,\beta,\tw^\beta)$ are Morita equivalent. By the version of the Packer-Raeburn stabilization trick given in \cite{Echterhoff:Morita_twisted},
 we know that  every $(G,N)$-twisted action is Morita equivalent to a twisted $(G,N)$-action with trivial twist, \ie, to a $G/N$-action.

Given a twisted $(G,N)$-action $(B,\beta,\tw)$ and a Hilbert $B,G$-module $(\E,\gamma)$, the \emph{crossed-product module} (or \emph{descent}) $\E\rtimes_{\gamma,\tw} (G,N)$ is the Hilbert $B\rtimes_{\beta,\tw}(G,N)$-module defined as the completion of the space $\contc(G,\E,\tw)$ of all continuous functions $x\colon G\to \E$ with compact support mod $N$ satisfying $x(ns)=x(s)\tw_{n^{-1}}$ for all $s\in G$ and $n\in N$, endowed with the structure of a pre-Hilbert module over $\contc(G,B,\tw)$ given by:
$$x\cdot f|_t\defeq \int_{G/N}x(s)\gamma_s(f(s^{-1}t))\dd_{G/N}(tN),$$
$$\bbraket{x}{y}|_t\defeq \int_{G/N}\beta_{s^{-1}}(\braket{x(s}{y(ts)})\dd_{G/N}(tN).$$
Given a $G$-algebra $(A,\alpha)$, the crossed product $A\rtimes_{\alpha|}N$ (where $\alpha|$ denotes the restriction of $\alpha$ to $N$) carries a twisted $(G,N)$-action $(\tilde\alpha,\iota_N)$ given by $\tilde\alpha_t(f)|_n\defeq \delta(t)\alpha_t(f(t^{-1}nt))$, where $\delta(t)=\frac{\Delta_G(t)}{\Delta_{G/N}(tN)}$ for all $t\in G$, and $\iota_N$ is the canonical homomorphism $N\to \M(A\rtimes N)$, $\iota_N(n)(f)|_s=\alpha_n(f(n^{-1}s)$. Sometimes $(\tilde\alpha,\iota_N)$ is called the \emph{decomposition twisted action} of $(G,N)$ on $A\rtimes_{\alpha|}N$. There is a canonical isomorphism $A\rtimes_\pn N\rtimes_\pn (G,N)\cong A\rtimes_\pn G$ for $\pn=\un$ or $\pn=\red$
(see \cite[Proposition 1]{Green:Local_twisted} and \cite[Proposition 5.2]{Kirchberg-Wassermann:permanence}).
Note that $(\tilde\alpha, \iota_N)$ factors through a twisted action (which we also denote $(\tilde\alpha, \iota_N)$)
 on a given exotic crossed product $A\rtimes_{\alpha|,\pn}N$
if and only if the ideal
$$I_\mu:=\ker(A\rtimes_{\alpha|}N\to A\times_{\alpha|,\pn}N)$$
 is $\tilde\alpha$-invariant. However, it is not clear how crossed-product norms for actions of $N, G$, and $G/N$
 should be related to each other in general to obtain the description of the $G$-crossed products as iterated crossed products.

We shall need twisted actions for the proof that for any weakly proper $X\rtimes G$-algebra
$(A,\alpha,\phi)$ we have a canonical isomorphisms
$$(A_u^N)_u^{G/N}\cong A_u^G \quad\text{and similarly}\quad (A_r^N)_r^{G/N}\cong A_r^G,$$
where $\|\cdot\|_u$ and $\|\cdot\|_r$ denote, as usual, the universal or reduced norms on crossed products by
$G$, $N$, and $G/N$, respectively. Note that in case where $G=N\times H$ is a direct product of groups, this
result has been shown in \cite[Lemma 5.17]{Buss-Echterhoff:Imprimitivity}.
Observe that by restricting the action $\alpha$  to $N$ provides us with the
weakly proper $X\rtimes N$-algebra $(A,\alpha|,\phi)$. We then denote by
$$\EE^N:A_c\to A_c^N; \;\EE^N(a)=\int_N^\st\alpha_s(a)\dd_N(s)$$
the corresponding  surjective "conditional expectation".

\begin{proposition}\label{prop:FixingN}
Suppose that  $(A,\alpha,\phi)$ is a weakly proper $X\rtimes G$-algebra and let $\|\cdot\|_\pn$ be any crossed-product norm
on $C_c(N,A)$ such that the corresponding  ideal $I_\mu\subseteq A\rtimes_{\alpha|}N$
is invariant under the decomposition action $\tilde\alpha$. Then the formula
$$\gamma^N_t(\xi)\defeq \delta(t)^{1/2}\alpha_t(\xi)$$
for $\xi\in \F^N_c(A)$ extends to a $G$-action $\gamma^N:G\to \Aut(\F_\mu^N(A))$ which is compatible with the decomposition $G$-action $\tilde\alpha$ of $G$ on $A\rtimes_{\alpha|, \pn}N$.
The corresponding $G$-action $\alpha^N\defeq\Ad\gamma^N$ on $A^N_\pn\cong \K(\F_u^N(A))$  is given  on the dense subalgebra $A_c^N$ by
the restriction of $\alpha$ to $A_c^N\subseteq\M(A)$ and satisfies the equation
\begin{equation}\label{eq-fixaction}
\alpha^N_t(\EE^N(a))=\alpha_t(\EE^N(a))=\delta(t)\EE^N(\alpha_t(a))\quad\forall a\in A_c, t\in G.
\end{equation}
This action is trivial on $N$ (hence is inflated from an action of $G/N$) and
$(\F_\pn^N(A), \gamma)$ is a correspondence  between the twisted actions $(A_\pn^N, \alpha^N, 1_N)$
and $(A\rtimes_{\alpha|,\pn}N, \tilde\alpha,\iota_N)$.
If  $N$ acts freely on $X$, this correspondence will be a $(G,N)$-equivariant Morita equivalence.
\end{proposition}
\begin{proof}
We know that $\F_\pn^N(A)\cong \F^N_X\otimes_{\contz(X)\rtimes_\tau N} A\rtimes_\pn N$ as Hilbert $A\rtimes_\pn N$-modules, via the map that sends $f\otimes g\in \contc(X)\odot \contc(N,A)$ to
$$f*g=\int_N\Delta_N(s)^{-1/2}\alpha_s(f\cdot g(s^{-1}))\dd_N(s)\in \contc(X)\cdot A=\F_c(A).$$
It was observed  in \cite[Remark 5.8]{Chabert-Echterhoff:Twisted} that $\F^N_X=\F^N(\contz(X))$ carries a $G$-action $\tilde\tau^N$, given by $\tilde\tau^N_t(f)=\delta(t)^{1/2}\tau_t(f)$ for all $f\in \contc(X)$, which is compatible with the twisted decomposition $(G,N)$-action on $\contz(X)\rtimes N$ and the $G/N$-action on $X_N:=N\bs X$. Hence, there is a $G$-action on $\F^N_X\otimes_{\contz(X)\rtimes_\tau N} A\rtimes_\pn N$ given on $\contc(X)\odot \contc(N,A)$ by the formula $\gamma^N_t(f\odot g)=\tilde\tau_t(f)\odot \tilde\alpha_t(g)$.
Using the isomorphism $\F_\pn^N(A)\cong \F^N_X\otimes_{\contz(X)\rtimes_\tau N} A\rtimes_\pn N$
we may view $\gamma^N$ as an action on $\F_\pn^N(A)$ and
 a straightforward computation shows that $\gamma_t(\xi)=\delta(t)^{1/2}\alpha_t(\xi)$ for all $\xi\in \F^N_c(A)$.
The corresponding action $\alpha^N\defeq \Ad\gamma^N$ on $A^N_\pn\cong \K(\F_\pn^N(A))$ is given, for all $\xi,\eta\in \F^N_c(A)$, by:
$$
\alpha^N_t(\EE^N(\xi\eta^*))=_{A^N_\pn}\bbraket{\gamma^N_t\xi}{\gamma^N_t\eta}=\delta(t) \EE^N(\alpha_t(\xi\eta^*))=\alpha_t(\EE^N(\xi\eta^*)),
$$
where the last equation follows from a straightforward computation using the fact that $\int_N \varphi(tnt^{-1})\,\dd n=\delta(t)\int_N \varphi(n)\,\dd n$ for
every integrable function $\varphi$ on $N$.
This proves  (\ref{eq-fixaction}). Since the elements in $A_c^N$ are fixed by $\alpha_n$ for all $n\in N$, it follows that $\alpha^N$ is trivial on $N$.

It is  straightforward  to check that $\gamma^N_n(\xi)*g=\xi*(\iota_N(n^{-1})\cdot g)$ for all $n\in N$, $\xi\in\F^N_c(A)$ and $g\in \contc(N,A)$ and hence $\gamma^N_n(\xi)=\xi\cdot \iota_N(n^{-1})$, which implies the compatibility of $\gamma^N$ with the twists.
The last assertion follows from the  fact that $\F_\pn^N(A)$ is a $A_\pn^N$--$A\rtimes_\pn N$ imprimitivity bimodule if
$N$ acts freely on $X$.
\end{proof}

If $(A,\alpha,\phi)$ is a weakly proper $X\rtimes G$-algebra  as in the proposition and if $N$ is a closed normal subgroup
of $G$, then $G/N$ acts properly on $N\bs X$ in a canonical way and  we
 observe that the homomorphism
 \begin{equation}\label{eq-phiN}
  \phi^N\colon \contz(N\bs X)\to \M(A^N_\pn)\cong\L_{A\rtimes N}(\F_\mu^N(A))
  \end{equation}
   induced by $\phi$ corresponds to the canonical left $\contz(N\bs X)$-action on $\F_\mu^N(A)\cong \F^N(X)\otimes_{\contz(X)\rtimes N}A\rtimes_\pn N$ given by
\begin{equation}\label{eq-leftaction}
\phi^N(f)m=\phi(f)m\quad\forall m\in A^G_c
\end{equation}
(where $\phi$ has been tacitly extended to $\contb(X)$ and we view $\contz(N\bs X)$ as a subalgebra of $\contb(X)$ in the usual way). From this it is easy to see that $\phi^N$ is $G/N$-equivariant and hence $(A^N_\pn,\alpha^N)$ is a weakly proper $N\bs X\rtimes G/N$-algebra.
Thus we may study the iterated fixed-point algebra $(A_\pn^N)_\pn^{G/N}$, if $\pn$ stands either for the universal or for the reduced
crossed-product norms. For the corresponding conditional expectation $\EE^{G/N}: (A_\mu^N)_c\to (A_\mu^N)_c^{G/N}$, we get the following:

\begin{lemma}\label{iterated-exp}
Let $(A,\alpha,\phi)$ and $N\subseteq G$ be as above. Then $\EE^N(A_c)$ is inductive limit dense in $(A_\mu^N)_c$ and hence
$\EE^{G/N}(\EE^N(A_c))$ is inductive limit dense in $(A_\mu^N)_c^{G/N}$. Moreover, we have
$$\EE^{G/N}(\EE^N(a))=\EE^G(a)\quad \forall a\in A_c.$$
Hence $A_c^G=\EE^G(A_c)=\EE^{G/N}(\EE^N(A_c))$ is inductive limit dense in $(A_\mu^N)_c^{G/N}$.
\end{lemma}
\begin{proof} The first assertion follows from the fact that $\EE^N$ and $\EE^{G/N}$ are continuous with respect to the
inductive limit topologies and that $A_c^N=\EE^N(A_c)$ is inductive limit dense in $(A_\mu^N)_c$. The second assertion
follows  from (\ref{iterated-int}).
\end{proof}

In what follows we abuse slightly the notation and
 write $\contc(G,\F^N_c(A),\iota_N)$ for the space of all functions $g\in \contc(G,\F_\pn^N(A),\iota_N)$  such
 that there exists an $f\in \contc(X)$ with $g(s)=f\cdot g(s)\in \F_c^N(A)$ for all $s\in G$.
 We leave it as an exercise for the reader to check that $\contc(G,\F^N_c(A),\iota_N)$ is inductive limit dense (with respect to compact supports mod $N$) in $\contc(G,\F_\pn^N(A),\iota_N)$.

\begin{proposition}\label{prop:FixingInStages}
Let $(A,\alpha,\phi)$ be a weakly proper $X\rtimes G$-algebra and consider the corresponding weakly proper $N\bs X\rtimes G/N$-algebra $(A^N_\pn,\alpha^N,\phi^N)$ as in Proposition~\ref{prop:FixingN}, where  $\pn$ stands either for the maximal or for the reduced crossed-product norms. Then there is a canonical isomorphism
$$\F_\pn^{G/N}(A^N_\pn)\otimes_{A^N_\pn\rtimes G/N}\big(\F_\pn^N(A)\rtimes (G,N)\big)\congto \F_\pn^G(A)$$
as Hilbert $A\rtimes_\pn N\rtimes_\pn (G,N)\cong A\rtimes_\pn G$-modules via the map sending $a\otimes g\in A^N_c\odot \contc(G,\F^N_c(A),\iota_N)$ to $a*g\defeq \int_{G/N}\Delta_G(s)^{-1/2}\alpha_s(a\cdot g(s^{-1}))\dd_{G/N}(sN)$.
\end{proposition}
\begin{proof}
Let us first observe that $A^N_c=\contc(N\bs X)\cdot A^N_c$ is indeed a dense subspace of $\F_\pn^{G/N}(A^N_\pn)$, and that the function $s\mapsto h(s)\defeq \Delta_G(s)^{-1/2}\alpha_s(a\cdot g(s^{-1}))$ is constant on $N$-orbits (and has compact support mod $N$) so that the integral over $G/N$ defining $a*g$ makes sense and gives an element of $A$. In fact, since $g\in \contc(G,\F^N_c(A),\iota_N)$, we have $g(n^{-1}s^{-1})=g(s^{-1})\cdot \iota_N(n)=\Delta_G(n)^{1/2}\alpha_{n^{-1}}(g(s))$ and since $a$ is $N$\nb-invariant, this implies
$$h(sn)=\Delta_G(sn)^{-1/2}\alpha_{sn}(a\cdot g(n^{-1}s^{-1}))=\Delta_G(s)^{-1/2}\alpha_s(a\cdot g(s^{-1}))=h(s)$$
 for all $s\in G$ and $n\in N$. Now observe that $a*g\in \F_\pn^G(A)=\contc(X)\cdot A$. In fact, take $K\sbe G$ compact such that $\supp(h)\sbe KN$ and let $\varphi\in \contc(G)$ with $\int_N\varphi(sn)\dd_N(n)=1$ for all $s\in K$. Then
$$\int_{G/N}h(s)\dd_{G/N}(sN)=\int_{G/N}\int_N h(sn)\varphi(sn)\dd_N(n)\dd_{G/N}(sN)=\int_G h(s)\varphi(s)\dd{s}.$$
Now, take $f\in \contc(X)$ with $g(s)=f\cdot g(s)$ for all $s\in G$. Since $a\in A^N_c$, we have $a\cdot f\in A_c=\contc(X)\cdot A\cdot \contc(X)$, so there is $\psi\in \contc(X)$ with $a\cdot f=\psi\cdot b$ for some $b\in A$. But then
$$a*g=\int_G \varphi(s)\Delta_G(s)^{-1/2}\tau_s(\psi)\cdot\alpha_s(b\cdot g(s^{-1}))\dd{s}$$
which is easily seen to be an element of $\F_c(A)$ (compare this to the formula~(2.10) in \cite{Buss-Echterhoff:Exotic_GFPA}). We therefore have a well-defined linear map from the dense subspace $A^N_c\odot \contc(G,\F^N_c(A),\iota_N)$ of $\F_\pn^{G/N}(A^N_\pn)\otimes_{A^N_\pn\rtimes G/N}\F_\pn^N(A)\rtimes (G,N)$ into the dense subspace $\F_c(A)$ of $\F(A)$. Now, a computation as in the proof of \cite{Buss-Echterhoff:Exotic_GFPA}*{Proposition~2.9} shows that this map preserves inner products and has dense range and hence extends to an isomorphism $\F_\pn^{G/N}(A^N_\pn)\otimes_{A^N_\pn\rtimes G/N}\F_\pn^N(A)\rtimes (G,N)\congto \F_\pn^G(A)$.
\end{proof}

Using the above proposition, we are now able to show the desired isomorphism $(A_\pn^N)_\pn^{G/N}\cong A_\pn^G$
for $\pn=u$ and $\pn=r$.
For reduced norms and {\em free} proper actions, this  result has been obtained in
\cite{Huef-Kaliszewski-Raeburn-Williams:Fixed}*{Theorem~4.5}.
Our  method of proof is, however, quite different from \cite{Huef-Kaliszewski-Raeburn-Williams:Fixed} and works for reduced and universal norms as well as for non-free proper actions.

\begin{theorem}\label{thm:FixingStages}
For a weakly proper $X\rtimes G$-algebra $A$, there is an isomorphism $(A^N_\pn)^{G/N}_\pn\cong A^G_\pn$
extending the inclusion map $A_c^G=\EE^{G/N}(A^N_c) \subseteq (A_\pn^N)_\pn^{G/N}$ into $A^G_c\subseteq A_\pn^G$ (where $\|\cdot\|_\pn$ denotes either the universal or the reduced crossed-product norm).
\end{theorem}
\begin{proof}
Let $\psi\colon \E\congto\F_\pn^G(A)$, $\psi(a\otimes g)=a*g$, be the isomorphism of Proposition~\ref{prop:FixingInStages}, where $\E\defeq \F_\pn^{G/N}(A^N_\pn)\otimes_{A^N_\pn\rtimes G/N}\big(\F_\pn^N(A)\rtimes_\pn (G,N)\big)$. This isomorphism induces an isomorphism $\widetilde\psi\colon \K(\E)\congto\K(\F_\pn^G(A))=A^G_\pn$ determined by the equation
$$\widetilde\psi(T)(\psi(a\otimes g))=\psi(T(a\otimes g)).$$
On the other hand, since $A^N_\pn\rtimes G/N\cong\K(\F_\pn^N(A)\rtimes_\pn(G,N))$, we have a canonical isomorphism $(A^N_\pn)^{G/N}_\pn=\K(\F_\pn^{G/N}(A^N_\pn))\congto \K(\E)$ sending an operator $S\in \K(\F_\pn^{G/N}(A^N_\pn))$ to the operator $S\otimes 1\in \K(\E)$ given by $(S\otimes 1)(a\otimes g)=S(a)\otimes g$.

We therefore get an isomorphism $(A^N_\pn)^{G/N}_\pn\congto A^G_\pn$ sending $S\in (A^N_\pn)^{G/N}_\pn$ to $\widetilde\psi(S\otimes 1)\in A^G_\pn$. Applying this to $S=\EE^{G/N}(b)=\EE^G(c)$ for $b=\EE^N(c)$, $c\in A_c$,
we see that $S$ is $G$-invariant. Hence we get:
\begin{align*}
\tilde\psi(S)(\psi(a\otimes g))&=\psi(Sa\otimes g)=Sa*g\\
&=\int_{G/N}\Delta_G(s)^{-1/2}\alpha_s(Sa\cdot g(s^{-1})\dd_{G/N}{sN}\\
&=S\left(\int_{G/N}\Delta_G(s)^{-1/2}\alpha_s(a\cdot g(s^{-1})\dd_{G/N}{sN}\right)\\
&=S\cdot (a*g)=S\cdot \psi(a\otimes g)
\end{align*}
so that $\tilde\psi\colon (A^N_\pn)^{G/N}_\pn\to A^G_\pn$ is the extension of the identity map on $\EE^{G/N}(A^N_c)$.
\end{proof}

Suppose  that $H$ is a closed subgroup of a group $G$ and that $(A,\alpha)$ is an $H$-algebra.
Then $C_0(G)\otimes A$ becomes a
weak $G\rtimes (G\times H)$ algebra with respect to the  structure map $\psi:f\mapsto f\otimes 1$ from $C_0(G)$
into $\cZ\M(C_0(G)\otimes A)$. We let $G$ act on $C_0(G)\otimes A$ via $\tau\otimes\id$, where
$\tau$ denotes the left translation action of $G$ on itself, and we let $H$ act on $C_0(G)\otimes A$ via
$\sigma\otimes\alpha$, with $\sigma_h(f)(s)=f(sh)$ for all $f\in C_0(G)$, $s\in G, h\in H$.
The actions of $G$ and $H$ on $C_0(G)\otimes A$ clearly commute and make the structure map $\psi$ equivariant with respect to both $G$ and $H$ actions.
Since $H$ acts properly on $G$, the restriction of the action to $H$ gives $C_0(G)\otimes A$ the structure
of a proper $G\rtimes H$-algebra and we can form the $H$-fixed-point algebras
$(C_0(G)\otimes A)_\pn^H$ for this structure. Since the structure map $\psi$ takes its values in
the center $\cZ\M(C_0(G)\otimes A)$, it follows from
 \cite[Theorem~3.28]{Buss-Echterhoff:Imprimitivity} that they do not depend on the given crossed-product norm $\|\cdot\|_\pn$
on $C_c(H,C_0(G)\otimes A)$ (indeed, for centrally proper actions all such norms coincide with  the universal norm $\|\cdot\|_u$).
Note that the $G$-action on $C_0(G)\otimes A$ factors to a $G$-action on the $H$-fixed-point algebra $(C_0(G)\otimes A)^H$
(we may now omit the norm $\pn$ in the notation).

The algebra $(C_0(G)\otimes A)^H$ is actually well-known under the name of
{\em induced algebra} $\Ind_H^G(A,\alpha)$ and can be described as follows:
$$
\Ind_H^G(A,\alpha)=\left\{F\in C_b(G,A): \begin{matrix} F(sh)=\alpha_{h^{-1}}(F(s))\;\forall s\in G, h\in H\\
\text{and}\;  \big(sH\mapsto \|F(s)\|)\in C_0(G/H)\end{matrix}\right\}.
$$
Indeed, identifying $\M(C_0(G)\otimes A)$ with the set of bounded strictly continuous functions from $G$ to $A$, it is an
easy exercise to check that $(C_0(G)\otimes A)_c^H$ is just the set of functions in $\Ind_H^G(A,\alpha)$ which have
compact supports mod $H$. In this picture, the $G$-action is given as the {\em induced action}
$$\Ind\alpha:G\to \Aut(\Ind_H^G(A,\alpha)); \Ind\alpha_s(F(t))=F(s^{-1}t)\quad \forall s,t\in G.$$
If $A=C_0(Y)$ for an $H$-space $Y$, we get $\Ind_H^GC_0(Y)\cong C_0(G\times_HY)$, where the {\em induced $G$-space}
$G\times_HY$ is defined as the quotient $H\bs(G\times Y)$ under the action of $H$ on $G\times Y$ given by
$h(s,y)=(sh^{-1}, hy)$. Moreover, if we start with a weakly proper $Y\rtimes H$-algebra
$(A,\alpha,\phi)$, the algebra $C_0(G)\otimes A$ actually becomes a weakly proper $(G\times Y)\rtimes(G\times H)$-algebra
via the obvious structure map $\psi:C_0(G\times Y)\to \M(C_0(G)\otimes A)$.
It follows then from \cite[Proposition 3.12]{Buss-Echterhoff:Imprimitivity} that the $H$-fixed-point algebras $(C_0(G)\otimes A)_\pn^H$
and the corresponding modules $\F_\pn(C_0(G)\otimes A)$ coincide,  no matter whether we regard $C_0(G)\otimes A$ as
a weakly proper $(G\rtimes Y)\rtimes H$ or a weakly proper $G\rtimes H$-algebra. But if we view it as a
$(G\times Y)\rtimes (G\times H)$-algebra, we see that $\Ind_H^G(A,\alpha)=(C_0(G)\otimes A)^H$ carries the
structure of a weakly proper $(G\times_HY)\rtimes G$-algebra, and by Theorem~\ref{thm:FixingStages} we obtain
isomorphisms
\begin{equation}\label{eq-isoinduced}
\begin{split}
(\Ind_H^G(A,\alpha))_\pn^G &\cong ((C_0(G)\otimes A)^H)_\pn^G\cong
(C_0(G)\otimes A)_\mu^{G\times H}\\
&\cong ((C_0(G)\otimes A)^G)_\pn^H\cong A_\mu^H,
\end{split}
\end{equation}
where $\|\cdot\|_\pn$ denotes either the universal or the reduced crossed-product norm (everywhere).
The last isomorphism in (\ref{eq-isoinduced}) is induced by the $H$-equivariant isomorphism
$A\cong \Ind_G^G(A,\id)$; $a\mapsto 1_G\otimes a$. We summarize our discussion as follows:

\begin{proposition}\label{prop:IndPreservesFix}
Let $H$ be a closed subgroup of $G$ and let $(A,\alpha,\phi)$ be a weakly proper $Y\rtimes H$-algebra.
Let $\|\cdot\|_\pn$ denote either the universal or reduced crossed-product norm for both $G$ and $H$.
Then there is an isomorphism
\begin{equation}\label{eq:IndPreservesFix}
A^{H,\alpha}_\pn\congto (\Ind_H^G(A,\alpha))^{G,\Ind\alpha}_\pn,
\end{equation}
sending $m\in A^{H,\beta}_c\sbe \M(A)^H$ to the constant function $t\mapsto m$ from $G$ to $\M(A)$ viewed as an element of
$\big(\Ind_H^G(A,\alpha)\big)^{G,\Ind\beta}_c\sbe \M(\contz(G)\otimes A)^G$.
\end{proposition}
\begin{proof}
The only statement which is not instantly clear from the above discussion is the special
description of the isomorphism $A_\pn^H\cong (\Ind_H^G(A,\alpha))_\pn^G$.
But this follows easily from the description of the isomorphism  $((C_0(G)\otimes A)^G)_\pn^H\cong A_\mu^H$
in (\ref{eq-isoinduced}) and the fact that, according to
Theorem  \ref{thm:FixingStages}, all other isomorphisms in (\ref{eq-isoinduced}) are induced
by the identity map on $(C_0(G)\otimes A)_c^{G\times H}$.
\end{proof}

\begin{remark}\label{rem:SpecialStructure}
Later we shall apply the above proposition to the special situation in which $Y=G$ equipped with the right translation
action of $H$. Let $(A,\alpha,\phi)$ be a $G\rtimes H$-algebra. Notice that the induced space $G\times_HG$ is $G$-homeomorphic to
$G/H\times G$ via $[(s,t)]\mapsto (sH, s t)$. If we forget the factor $G/H$, we see that $\Ind_H^G(A,\alpha)$
carries a structure of a weakly proper $G\rtimes G$-algebra with structure map $\psi:C_0(G)\to \M(\Ind_H^G(A,\alpha))$ given by
the formula
$$(\psi(f)F)(s)=\phi(\tau_{s^{-1}}(f))F(s)\quad\forall f\in C_0(G), F\in \Ind_H^G(A,\alpha).$$
It follows from \cite[Proposition 3.12]{Buss-Echterhoff:Imprimitivity} that the $G$-fixed-point algebra
 for this $G\rtimes G$-structure on $\Ind_H^G(A,\alpha)$
coincides with the $G$-fixed-point algebra for the $(G\times_HG)\rtimes G$-structure, hence Proposition \ref{prop:IndPreservesFix}
will still apply if we just consider the $G\rtimes G$-structure.
\end{remark}

Suppose now that $(A,\alpha,\phi)$ is a weakly proper $X\rtimes G$-algebra and that $H$ is a closed subgroup of $G$.
Then $(A,\alpha|,\phi)$, where $\alpha|$ denotes the restriction of $\alpha$ to $H$, is a weakly proper $X\rtimes H$-algebra,
and we close this section by proving an  isomorphism
$$\F_\pn^G(A)\otimes_{A\rtimes_\pn G}\X^G_{H,\pn}(A)\cong \F_\pn^H(A),$$
where $\|\cdot\|_\pn$ denotes either the universal or the reduced crossed-product norm.
Here $\X^G_{H,\pn}(A)$ denotes Green's $C_0(G/H,A)\rtimes_\pn G$--$A\rtimes_\pn H$ imprimitivity bimodule of \cite[\S 2]{Green:Local_twisted}.
Recall that it is the completion of $C_c(G,A)$ viewed as a $C_c(H,A)$-pre-Hilbert module with respect to the
module action and inner product
 given by the formulas
\begin{align*}
\xi\cdot\varphi(t)&=\int_H\gamma_H(h)\xi(th) \alpha_{th}\big(\varphi(h^{-1})\big)\dd{h}\\
\bbraket{\xi}{\eta}_{\contc(H,A)}(h)
&=\gamma_H(h)\int_G\alpha_{s^{-1}}\big(\xi(s)^*\eta(sh))\big)\dd{s}
\end{align*}
where $\gamma_H(h):=\sqrt{\Delta_G(h)\Delta_H(h^{-1})}$ for $h\in H$. The above formulas are taken from
\cite[Theorem 4.15]{Williams:crossed-products}. The left action of $C_c(G,C_0(G/H,A))\subseteq C_0(G/H,A)\rtimes_\pn G$ on $C_c(G,A)$ is given by the
formula
$$f*\xi(s)=\int_G f(t,sH)\alpha_t(\xi(t^{-1}s))\,\dd t.$$
The $G$-equivariant  imbedding of $A$ into $\M(C_0(G/H,A))$ as constant functions induces the left action
of $A\rtimes_\pn G$ on $\X^G_{H,\pn}(A)$ given on the level of functions on $f,\xi\in C_c(G,A)$ by
the usual convolution $(f\cdot \xi)(s)=f*\xi(s)=\int_G f(t)\alpha_t(\xi(t^{-1}s))\,\dd t$.
Combining these formulas with the formulas for the pre-imprimitivity modules
$\F_c^G(A)$ and $\F_c^H(A)$ as given in \S \ref{sec-prel} we get the following:

\begin{proposition}\label{prop:inductionViaGFPA}
 Let $A$ be a weakly proper $X\rtimes G$-algebra and let $H$ be a closed subgroup of $G$.
Let $\pn=\un$ or $\pn=\red$. Then there is an isomorphism
$$\F^G_\pn(A)\otimes_{A\rtimes_{\alpha,\pn}G}\X^G_{H,\pn}(A)\congto \F^H_\pn(A)$$
of Hilbert $A\rtimes_{\beta|,\pn}H$-modules which sends an elementary tensor $a\otimes\xi$ in $\F^G_c(A)\odot \contc(G,A)$ to
$a*\xi\defeq \int_G\Delta_G(t)^{-1/2}\alpha_t(a\cdot \xi(t^{-1}))\dd{t}$.

The induced \Star{}homomorphism on compact operators $A^G_\pn\to \M(A^H_\pn)$ is
given by the identity map on $A^G_c$ by viewing $m\in A_c^G$ as a multiplier of $A^H_\pn$
via multiplication in $\M(A)$: $a\cdot m=am$ and $m\cdot a=ma$ for all $a\in A^H_c$.
\end{proposition}
\begin{proof}
It is enough to check that the map $a\otimes\xi\mapsto a*\xi$ preserves inner products and has dense range.
For the inner products we let $a,b\in \F_c(A)$ and $\xi,\eta \in C_c(G,A)$  we compute
\begin{align*}
&\bbraket{a\otimes \xi}{b\otimes \eta}_{C_c(H,A)}(h)=
\bbraket{\xi}{\bbraket{a}{b}_{C_c(G,A)}\cdot \eta}_{C_c(H,A)}(h)\\
&=\gamma_H(h)\int_G\int_G \Delta_G(t)^{-1/2} \alpha_{s^{-1}}\big(\xi(s)^*\alpha_h\big(a^* \alpha_t(b\eta(t^{-1}sh))\big)\big)\,\dd t\,\dd s.
\end{align*}
On the other hand, we compute
\begin{align*}
&\bbraket{a*\xi}{b*\eta}_{C_c(H,A)}(h)=\Delta_H(h)^{-1/2}(a*\xi)^*\alpha_h(b*\eta)\\
&=\Delta_H(h)^{-1/2}\int_G\int_G \Delta_G(st)^{-1/2}\alpha_s\big(\xi(s^{-1})^*a^*)\alpha_{ht}\big(b\xi(t^{-1})\big)\,\dd t\,\dd s
\end{align*}
If we apply the transformation $s\mapsto s^{-1}$ followed by the transformation $t\mapsto h^{-1}s^{-1}t$ to the above integral, we see
 that $\bbraket{a\otimes \xi}{b\otimes \eta}_{C_c(H,A)}(h)=\bbraket{a*\xi}{b^*\eta}_{C_c(H,A)}(h)$ for all $h\in H$.
A similar, but easier computation shows that $a*(\xi\cdot\varphi)=(a*\xi)\cdot\varphi$ for
all $a\in \F_c^G(A), \xi\in C_c(G,A)$ and $\varphi\in C_c(H,A)$,  which then implies
that the map $a\otimes \xi\mapsto a*\xi$ extends to an isometric $A\rtimes_{\pn}H$-Hilbert-module map from
$\F^G_\pn(A)\otimes_{A\rtimes_{\alpha,\pn}G}\X^G_{H,\pn}(A)$ into $\F_\pn^H(A)$. Surjectivity of this map
follows from standard approximative unit arguments as done, for example,
 in the proof of \cite{Buss-Echterhoff:Imprimitivity}*{Proposition~3.32}.
For the final assertion, it is enough to check that given $m\in A^G_c$, we have $(m\cdot a)*\xi=m\cdot (a*\xi)$ for all $a\in \F^G_c(A)$ and $\xi\in \contc(G,B)$. This
 follows from a simple computation using that $m$ is $G$-invariant.
\end{proof}

Observe that the canonical homomorphism $A^G_\pn\to \M(A^H_\pn)$ can be used to induce representations from $A^H_\pn$ to $A^G_\pn$.
The above proposition says that this corresponds to the induction process from representations of $A\rtimes_\pn H$ to representations of $A\rtimes_\pn G$
 via Green's imprimitivity bimodule.
This is especially interesting if the involved actions are saturated, in which case $\F^G_\pn(A)$ and $\F^H_\pn(A)$ are imprimitivity bimodules implementing Morita equivalences
$A^G_\pn\sim A\rtimes_\pn G$ and $A^H_\pn\sim A\rtimes_\pn H$, so that we get bijections between the spaces of representations
$\Rep(A^G_\pn)\cong \Rep(A\rtimes_\pn G)$ and $\Rep(A^H_\pn)\cong \Rep(A\rtimes_\pn H)$. In this situation the induction process
 $\Rep(A\rtimes_\pn H)\to \Rep(A\rtimes_\pn G)$ is therefore essentially equivalent to $\Rep(A^H_\pn)\to\Rep(A^G_\pn)$, but the latter might be easier to describe in some situations.

\section{Mansfield's Imprimitivity Theorem}\label{sec-Mansfield}

As a consequence of our previous results, we deduce Mansfield's Imprimitivity Theorem for both universal and reduced crossed-product norms.
So in what follows next we let $\delta:B\to M(B\otimes C^*(G))$ be a coaction of $G$ on the \cstar{}algebra $B$. Recall from
\S \ref{sec-prel} that $B\rtimes_{\delta}\dualG$ carries a canonical weakly proper $G\rtimes G$-structure $(B\rtimes_\delta\dualG, j_G,\widehat{\delta})$
which restricts to a $G\rtimes H$-structure $(B\rtimes_\delta\dualG, j_G,\widehat{\delta}|_H)$ for every closed subgroup $H$ of $G$.
Since right translation of $H$ on $G$ is free, we see that $\F_\pn^H(B\rtimes_\delta\dualG)$ implements a
$(B\rtimes_{\delta}\dualG)_\pn^H$--$(B\rtimes_{\delta}\dualG)\rtimes_\pn H$ imprimitivity bimodule for every crossed-product norm
$\|\cdot\|_\pn$ on $C_c(H, B\rtimes_{\delta}\dualG)$.

In what follows, we want to compare this result with the various versions of Mansfield's Imprimitivity Theorem for coactions which
give rise to Morita equivalences between $B_\pn\rtimes_{\delta_\pn|}\widehat{G/N}$ and $(B\rtimes_{\delta}\dualG)\rtimes_\pn N$,
where the notation $(B_\pn,\delta_\mu)$ indicates that we have to be careful about the type of coactions we may consider here.
Indeed, we shall restrict below to the two cases where $\|\cdot\|_\pn$ is either the universal norm $\|\cdot\|_u$ or the reduced
norm $\|\cdot\|_r$. Then, as explained in \S \ref{sec-prel}, $(B_u,\delta_u)$ is the maximalization and
$(B_r,\delta_r)$ is  the normalization of $(B,\delta)$. Recall that for any coaction $\delta:B\to M(B\otimes C^*(G))$, the restriction
$\delta|$ of $\delta$ to the quotient group $G/N$ is given by the composition
$$\delta|: B\stackrel{\delta}{\longrightarrow} M(B\otimes C^*(G))\stackrel{\id_B\otimes q_N}{\longrightarrow} M(B\otimes C^*(G/N)),$$
where $q_N:C^*(G)\to C^*(G/N)$ denotes the canonical quotient map.

\begin{theorem}[Mansfield's Imprimitivity Theorem]\label{thm:GPFA(BxG)N}
Let $(B,\delta)$ be a $G$-coaction and equip the crossed product $B\rtimes_\delta\dualG$ with the canonical weak $G\rtimes G$\nb-algebra structure. Then there are canonical
isomorphisms
$$B_r\rtimes_{\delta_r|}\widehat{G/N}\cong (B\rtimes_\delta\dualG)_r^N\quad\text{and}\quad
B_u\rtimes_{\delta_u|}\widehat{G/N}\cong (B\rtimes_\delta\dualG)_u^N.$$
In particular, if $(B,\delta)$ is normal, $\F_r^N(B\rtimes_\delta\dualG)$ becomes a
$B\rtimes_{\delta|}\widehat{G/N}$--$(B\rtimes_\delta\dualG)\rtimes_rN$ imprimitivity bimodule and if
$(B,\delta)$ is maximal, then $\F_u^N(B\rtimes_{\delta}\dualG)$ becomes a
$B\rtimes_{\delta|}\widehat{G/N}$--$(B\rtimes_\delta\dualG)\rtimes_uN$ imprimitivity bimodule.
\end{theorem}

\begin{remark}\label{rem-Mans}
We should remark that the isomorphism for normal coactions has been established before
by Quigg and Raeburn in \cite[Proposition 4.1]{Quigg-Raeburn:Induced} in case where $N$ is amenable, and
shortly after that by Kaliszewski and Quigg in \cite{Kaliszewski-Quigg:Mansfield} for arbitrary closed
normal subgroups $N$. Both proofs rely heavily on Mansfield's original proof of his imprimitivity theorem
and they use the Mansfield algebra $\mathcal D\subseteq B\rtimes_{\delta}\dualG$ as a dense
\Star{}algebra which implements properness in Rieffel's sense (\cite{Rieffel:Proper}). So a priori, the fixed-point algebras
and the bimodules considered in those papers could be different from ours, but we shall see below that they are not.

The isomorphism for maximal coactions is new but a version of Mansfield's Imprimitivity Theorem
for maximal coactions has been shown by   Kaliszewski and Quigg in \cite{Kaliszewski-Quigg:Mansfield}
using quite different techniques. The above gives a unified treatment to all of these different versions and
does not rely on Mansfield's techniques.
\end{remark}

\begin{proof}[Proof of Theorem~\ref{thm:GPFA(BxG)N}] In what follows let $A:=B\rtimes_\delta\dualG$. It follows then from
 \cite[Theorem 4.6]{Buss-Echterhoff:Exotic_GFPA} that  $(B_u,\delta_u)\cong (A_u^G,\delta_u^G)$ and $(B_r,\delta_r)\cong
 (A_r^G,\delta_r^G)$  where the coactions $\delta_\pn^G$, with $\pn=u$ or $\pn=r$, are given by the formulas
 \begin{equation}\label{eq-coaction}
 \delta^G_\pn(m)=(j_G\otimes\id)(\omega_G)(m\otimes 1)(j_G \otimes\id)(\omega_G)^*\quad\mbox{for all }m\in A^G_c
 \end{equation}
 (see also \cite[Remark 4.14]{Buss-Echterhoff:Exotic_GFPA} for the correct interpretation of this formula).
  On the other hand, $A^N_\pn$ is a weak $G/N\rtimes G/N$-algebra and by
 Proposition \ref{prop:FixingInStages} we have a canonical isomorphism $\Psi :A^G_\pn\stackrel{\cong}{\longrightarrow} (A^N_\pn)^{G/N}_\pn$ given via the
 canonical inclusion of $A_c^G$ into $(A_\pn^N)_c^{G/N}$.
Applying  \cite[Theorem 4.6]{Buss-Echterhoff:Exotic_GFPA} again, we see that  $(A^N_\pn)^{G/N}_\pn$ carries a $G/N$-coaction $\delta_\pn^{G/N}$
given by
$$\delta_\pn^{G/N}(n)=(j_G^N\otimes\id)(\omega_{G/N})(n\otimes 1)(j_G^N\otimes\id)(\omega_{G/N})^*\quad\mbox{for all }n\in (A^N_\pn)_c^{G/N},$$
where $j_G^N\colon\contz(G/N)\to \M(A^N_\pn)$
is the structural homomorphism induced from $j_G\colon \contz(G)\to \M(A)$. It is given by
the equation $j_G^N(\EE^{\tau,N}(f))=\EE^{\alpha,N}(j_G(f))$ for all $f\in \contc(G)$.
We claim that $\delta_\pn^{G/N}$ corresponds to the restriction $\delta_\pn^{G}|$ of $\delta_\mu^G$ to the quotient $G/N$
via the isomorphism $A^G_\pn\cong (A^N_\pn)^{G/N}_\pn$.
Notice that $j_G^N$, once composed with the canonical homomorphism $\kappa\colon A^N_\pn\to \M(A)$,
coincides with the restriction of $j_G$ to $\contz(G/N)\sbe \contb(G)$, that is, $\kappa\circ j_G^N=j_G|_{\contz(G/N)}$.
Then, if $q_N\colon C^*(G)\to C^*(G/N)$ is the quotient map, we get
 $$(j_G\otimes q_N)(\omega_G)=(j_G\otimes\id)(\omega_{G/N})=(\kappa\circ j_G^N\otimes\id)(\omega_{G/N}),$$ so that
\begin{align*}
\delta^G_\pn|(m)&=(\id\otimes q)\circ \delta^G_\pn(m)=(\kappa\circ j_G^N\otimes\id)(\omega_{G/N})(m\otimes 1)(\kappa\circ j_G^N\otimes\id)(\omega_{G/N})^*\\
&=(\kappa\otimes\id)\big(\delta_\pn^{G/N}(m)\big),
\end{align*}
for $m\in A_c^G\subseteq (A_\pn^N)_c^{G/N}$, which proves the claim.

Now \cite[Theorem 4.6]{Buss-Echterhoff:Exotic_GFPA} applied to the weak $G/N\rtimes G/N$-algebra
$A_\pn^N$ gives an isomorphism $A^N_\pn\cong (A^N_\pn)^{G/N}_\pn\rtimes_{\delta^{G/N}_\pn}\dual{G/N}$ and if we combine
this with the $\widehat{G/N}$\nb-iso\-mor\-phisms $((A_\pn^N)_\pn^{G/N}, \delta_\pn^{G/N})\cong (A_\pn^G, \delta_\pn^G|)\cong (B_\pn,\delta_\pn)$
and the fact that $B_\pn\rtimes_{\delta_\pn}\dualG\cong B\rtimes_{\delta}\dualG$ for $\pn=u,r$, we finally obtain a chain of isomorphisms
$$(B\rtimes_\delta\dualG)_\pn^N=
A^N_\pn\cong (A^N_\pn)^{G/N}_\pn\rtimes_{\delta^{G/N}_\pn}\dual{G/N}\cong A_\mu^G\rtimes_{\delta_\mu|}\widehat{G/N}\cong
B_\mu\rtimes_{\delta_\mu|}\widehat{G/N}.$$
This finishes the proof.
\end{proof}

In what follows next, we want to compare our module $\F_\pn^H(B\rtimes_\delta\dualG)$ with Mansfield's original construction
in \cite{Mansfield:Induced} which provides us with
an explicit description of  a dense submodule of the $(B\rtimes_\delta\dualG)^H_\mu$--$(B\rtimes_\delta\dualG)\rtimes_\mu H$ bimodule
$\F_\mu^H(B\rtimes_\delta\dualG)$ and a subalgebra $\mathcal D^H$ sitting densely inside the fixed-point algebra
with compact supports $(B\rtimes_\delta\dualG)_c^H$ with respect to any chosen norm $\mu$ on $C_c(H, B\rtimes_\delta\dualG)$
as above.
\begin{notations}[{\cf~\cite{Mansfield:Induced}}]\label{not-Mansfield}
For a locally compact group $G$ we let $B(G)\cong C^*(G)^*$ denote the Fourier-Stieltjes algebra and we denote by
$A(G)\subseteq B(G)$ the Fourier algebra of $G$, \ie, the set of matrix coefficients of the regular representation $\lambda_G$ of $G$.
For  $w\in B(G)$ we let $\delta_w:B\to B$ denote the composition
$\delta_w(b)=(\id_B\otimes\, w)\circ \delta(b)$.
Let $A_c(G):=A(G)\cap C_c(G)\subseteq B(G)$.
For a compact subset $E\subseteq G$ we denote by $C_E(G)$ the set of functions $f\in C_c(G)$ with
support in $E$. Recall that $\EE^H: C_c(G)\to C_c(G/H)$ denotes the surjective linear map  given by
$$\EE^H(f)(gH)=\int_H f(gh)\;\dd h.$$
For $w\in A_c(G)$ and $E\subseteq G$ compact, let
 $$\mathcal D_{w,E, H}:=\overline{j_B(\delta_w(B))j_G(\EE^H(C_E(G)))}\subseteq \M(B\rtimes_\delta\dualG)$$
and
$$\mathcal D_H:=\bigcup\big\{\mathcal D_{(w,E, H)}: w\in A_c(G), E\subseteq G\;\text{compact}\big\}.$$
We call $\mathcal D_H$ the {\em $H$-Mansfield subalgebra of $\M(B\rtimes_\delta\dualG)$}.
We simply write $\mathcal D$ in case where $H=\{e\}$.
\end{notations}

We should note for later use that the general assumption on our coactions being nondegenerate -- in the
sense that $\delta(B)(1_B\otimes C^*(G))=B\otimes C^*(G)$ -- implies that
\begin{equation}\label{eq-Bc}
B_c:=\delta_{A_c(G)}(B)=\{\delta_w(B): w\in A_c(G)\}
\end{equation}
is norm dense in $B$. This follows from  \cite[Theorem 5]{Katayama:Takesaki_Duality} together with fact
that $A_c(G)$ is weak-* dense in $A(G)$.
Mansfield shows  the following:

\begin{lemma}\label{lem-Mansfield} Let $(B,\delta)$ be a coaction of $G$ and let $H$ be a closed subgroup of $G$.
Then
\begin{enumerate}
\item $\D_H$ is a dense \Star{}subalgebra of $C^*\big(j_B(B)j_G(C_0(G/H))\big)\subseteq \M(B\rtimes_\delta\dualG)$ and we have
$\D_H=j_G(C_c(G/H))\cdot\D_H=\D_H\cdot j_G(C_c(G/H))$.
\item  $\mathcal D$ is a dense \Star{}subalgebra of $B\rtimes_\delta\dualG$.
\end{enumerate}
\end{lemma}
\begin{proof}
The first assertion in (1) follows from \cite[Lemma 11]{Mansfield:Induced} and the equation
$\D_H=j_G(C_c(G/H))\cdot\D_H=\D_H\cdot j_G(C_c(G/H))$ is a consequence of \cite[Lemma 9]{Mansfield:Induced} (see also \cite{anHuef-Raeburn:Mansfield}*{Lemma~3.2}).
Note that Mansfield proved both lemmas in the full generality of arbitrary closed subgroups $H$ of $G$.
Item (2) follows from \cite[Theorem 12]{Mansfield:Induced} in the special case $H=\{e\}$.
\end{proof}

In the special case where $H=G$ we get the \Star{}algebra
$\D_G=\bigcup_{w\in A_c(G)} \overline{j_B(\delta_w(B))}$.
This \Star{}algebra has an easier description:

\begin{lemma}\label{lem-B0}
$B_c=\delta_{A_c(G)}(B)$ is a dense \Star{}subalgebra of $B$ and $j_B:B_c\to \D_G$ is an isomorphism of \Star{}algebras.
In particular, $\D_G=j_B(\delta_{A_c(G)}(B))$.
\end{lemma}
\begin{proof} We already observed above that $B_c=\delta_{A_c(G)}(B)$ is dense in $B$ and the fact that it is a \Star{}subalgebra
of $B$ follows from items (ii)-(iv) of \cite[Lemma 1]{Mansfield:Induced}. To see, for instance, that $B_c$ is a vector subspace, observe that, by \cite[Lemma 1(ii)]{Mansfield:Induced}, if $b_1,b_2\in B$, $v_1,v_2\in A_c(G)$ and $w\in A_c(G)$ is such that $w=1$ on $\supp(v_1)\cup\supp(v_2)$, then $\delta_{v_1}(a_1)+\delta_{v_2}(a_2)=\delta_{wv_1}(a_1)+\delta_{wv_2}(a_2)
=\delta_w(\delta_{v_1}(a_1)+\delta_{v_2}(a_2))$. Items (iii) and (iv) in \cite[Lemma 1]{Mansfield:Induced} imply that $B_c$ is  a \Star{}subalgebra of $B$.

To show that $j_B$ gives an \Star{}isomorphism $B_c\congto \D_G$, we first show that it is surjective, that is, $j_B(\delta_{A_c(G)}(B))=\D_G$.
For this assume that $w\in A_c(G)$ is fixed and that $(b_n)_n$
is a sequence in $B$ and $b\in B$ such that $j_B(\delta_w(b_n))\to j_B(b)$. It suffices to show that $j_B(b)=j_B(\delta_v(b))$
for some $v\in A_c(G)$.

We first note that $I=\ker j_B$ is annulated
by $\delta_w:B\to B$ for any $w\in A(G)$. This follows from the fact that the kernel $\ker\lambda_G\subseteq C^*(G)$
is annulated by the elements in $A(G)$ viewed as linear functionals of $C^*(G)$. Indeed, if we realize
$B\rtimes_\delta\dualG$ as a subalgebra of $\M(B\otimes \K(L^2G))$ via the covariant representation
\mbox{$((\id\otimes\lambda_G)\circ \delta, 1\otimes M)$}, we see that $\ker j_B=\ker(\id_B\otimes\lambda_G)\circ \delta$.
Moreover, since for $w\in A(G)$ the linear functional on $C^*(G)$ associated to $w$ factors through a functional $w_r$ of
$C_r^*(G)$, we see that
$\delta_w$ is given by the
composition
$$\delta_w=(\id_B\otimes \,w_r)\circ (\id_B\otimes \lambda_G)\circ \delta.$$
Hence $\ker j_B=\ker (\id_B\otimes\lambda_G)\circ \delta\subseteq \ker \delta_w$.
Suppose now that
 $j_B(\delta_w(b_n))\to j_B(b)$. By passing to a subsequence, if necessary,
we may choose elements $c_n\in \ker j_B$ with $\delta_w(b_n)+c_n\to b$ in $B$.
Now let $v\in A_c(G)$ such that $v\equiv 1$ on $\supp w$. Then
$$\delta_w(b_n)=\delta_{vw}(b_n)=\delta_v(\delta_w(b))=\delta_v(\delta_w(b_n)+c_n)\to \delta_v(b),$$
and hence $j_B(\delta_w(b_n))\to j_B(\delta_v(b))$, which proves that $j_B(b)=j_B(\delta_v(b))\in j_B(B_c)$,
hence $j_B(B_c)=\D_G$.

We now use item (ii) of \cite[Lemma 1]{Mansfield:Induced}, that is, the fact that $\delta_{wv}(b)=\delta_w(\delta_v(b))$ for all $w,v\in A_c(G)$,
to show that $j_B:B_c\to \D_G$ is injective and hence an isomorphism of \Star{}algebras. For this assume that $j_B(\delta_w(b))=0$
for some $w\in A_c(G)$ and $b\in B$. Let $v\in A_c(G)$ such that $vw=w$. Since $\ker j_B\subseteq \ker \delta_v$,
we get $0=\delta_v(\delta_w(b))=\delta_{vw}(b)=\delta_w(b)$, and the result follows.
\end{proof}

For later use, we should also note that $\D_H$ is a bimodule over $B_c\cong \D_G$
with bimodule operations given by the usual multiplication inside $\M(B\rtimes_\delta\dualG)$, that is,
\begin{equation}\label{eq-multiply}
\delta_w(b)\cdot d=j_B(\delta_w(b))d\quad\text{and}\quad d\cdot\delta_w(b)=d j_B(\delta_w(b)),
\end{equation}
for $d\in \D_H$. This gives a canonical imbedding of $B_c=\delta_{A_c(G)}(B)$ into the (algebraic) multiplier
algebra $\M(\D_H)$.

Recall from \cite{Buss-Echterhoff:Exotic_GFPA} that for any weakly proper $G\rtimes H$-algebra $A$
and for any given crossed-product norm $\|\cdot\|_\pn$ on $C_c(H,A)$, the
$A_\pn^H$--$A\rtimes_{\pn}H$-imprimitivity bimodule $\F_\pn(A)$ is given as the completion of the
pre-imprimitivity $A_c^H$--$C_c(H,A)$ bimodule  $\F_c(A):= C_c(G)\cdot A$.
Moreover, we showed in \cite[Lemma 2.12]{Buss-Echterhoff:Exotic_GFPA} that all bimodule operations are
continuous with respect to the inductive limit topologies on $A_c^H$, $C_c(H,A)$ and $\F_c(A)$, respectively,
and that in all three spaces, inductive limit convergence implies norm-convergence in their
respective completions for any chosen norm $\|\cdot\|_\pn$.
Similarly, the canonical "conditional expectation"
$$\EE^H: A_c\to A_c^H,\quad \EE^H(x)=\int_H^{st}\alpha_t(x)\,\dd t,$$
is inductive limit continuous on $A_c=C_c(G)\cdot A\cdot C_c(G)$ and we have $A_c^H=\EE^H(A_c)$.
Recall also that the inductive limit topology on $C_c(H,A)$ is the usual one, and that a net $(a_i)_{i\in I}$
in $\F_c(A)$  (resp. in $A_c$)
 converges in the inductive limit topology to some $a$ if it converges  to $a$ in norm and there exists an $f\in C_c(G)$
 such that $a_i=f\cdot a_i$ (resp. $a_i=f\cdot a_i\cdot f$) for all $i\in I$.
 Similarly, a net $(b_i)_{i\in I}$ in $A_c^H$ converges to $b\in A_c^H$ in the inductive limit topology, if
it converges in $\M(A)$ in norm and if the following are satisfied
\begin{enumerate}
\item there exists a $\psi\in C_c(G/H)$ such that $\psi\cdot b_i\cdot\psi=b_i$ for all $i\in I$, and
\item for all $f\in C_c(G)$ the net $b_i\cdot f$ converges to $b\cdot f$ in the inductive limit topology of $A_c$.
\end{enumerate}
(recall that $b\cdot f\in A_c$ for all $b\in A_c^H$ and $f\in C_c(G)$, where the multiplication is performed inside $\M(A)$).

Recall that we always use the notation $f\cdot a$ for $\phi(f)a$ if $\phi:C_0(G)\to \M(A)$ is the given structure map.
In case where $A=B\rtimes_\delta\dualG$, this structure map is given by the canonical map
$j_G:C_0(G)\to \M(B\rtimes_\delta\dualG)$. Thus we get the following:

\begin{lemma}\label{lem-compact supports}
Let $(B,\delta)$ be a coaction of $G$ and let $H$ be a closed subgroup of $G$. Then
\begin{enumerate}
\item  $\D$ is inductive limit dense in both  $(B\rtimes_\delta\dualG)_c$ and $\F_c(B\rtimes_\delta\dualG)$.
\item  $\D_H$ is inductive limit
 dense in $(B\rtimes_\delta\dualG)_c^H$.
 \end{enumerate}
 In particular, every generalized fixed-point algebra
 $(B\rtimes_\delta\dualG)_\mu^H$ is a norm completion of $\D_H$ for some suitable norm.
\end{lemma}
\begin{proof}
Since $\D$ is norm dense in $B\rtimes_\delta\dualG$ it follows that
$\D=j_G(C_c(G))\cdot \D\cdot j_G(C_c(G))$ is inductive limit dense in $j_G(C_c(G))\cdot (B\rtimes_\delta\dualG)\cdot j_G(C_c(G))=(B\rtimes_\delta\dualG)_c$
and a similar argument shows density in $\F_c(B\rtimes_\delta\dualG)$.
To check that $\D_H$ is a subalgebra of  $(B\rtimes_\delta\dualG)_c^H$, we first observe that $\D_H$ lies in the
(classical) fixed-point algebra $\M(B\rtimes_\delta\dualG)^H$. Moreover, multiplying
$\D_H=j_G(C_c(G/H))\cdot \D_H\cdot j_G(C_c(G/H))$ with $j_G(f)$ for some $f\in C_c(G)$ from either side
gives an element in $\D\subseteq (B\rtimes_\delta\dualG)_c$. Thus it follows easily from the definition of the
fixed-point algebra with compact supports as given in (\ref{eq:GFPAwithCompSupp}) that
$\D_H\subseteq (B\rtimes_\delta\dualG)_c^H$.
Since
$$_{(B\rtimes_\delta\dualG)^H_c}\bbraket{\D}{\D}\subseteq \EE(\D)\subseteq \D_H$$
and since $_{(B\rtimes_\delta\dualG)^H_c}\bbraket{\D}{\D}$ is inductive limit dense in
$(B\rtimes_\delta\dualG)^H_c$ due to the facts that $\D$ is inductive limit dense in $(B\rtimes_\delta\dualG)_c$
and that every element in $(B\rtimes_\delta\dualG)^H_c$ can be written as an inner product
of two elements in $(B\rtimes_\delta\dualG)_c$, it  follows that $\D_H$ is inductive limit dense
in $(B\rtimes_\delta\dualG)^H_c$. The final assertion now follows from the fact that inductive limit convergence
implies norm convergence in $(B\rtimes_\delta\dualG)^H_\pn$ with respect to any given crossed-product norm $\|\cdot\|_\pn$
on $C_c(H, B\rtimes_\delta\dualG)$.
\end{proof}

\begin{remark}\label{rem-Bc} In the case $H=G$, the above result shows that for any coaction $(B,\delta)$ of $G$,
the dense \Star{}subalgebra $B_c=\delta_{A_c(G)}(B)$ of $B$ maps faithfully onto the inductive limit dense
\Star{}subalgebra $j_B(B_c)=\D_G$ of $(B\rtimes_\delta\dualG)_c^G$.
It follows then from the above lemma that  for a given
norm $\|\cdot\|_\pn$
on $C_c(G, B\rtimes_\delta\dualG)$ the
$\pn$-fixed-point algebra $B_\pn:=(B\rtimes_\delta\dualG)_\pn^G$ can be obtained
as a completion of $B_c\cong \D_G$ with respect to a suitable norm induced from
$\|\cdot\|_\mu$ via the bimodule $\F_\mu(B\rtimes_\delta\dualG)$. In this picture,
the canonical epimorphisms
$$B_u\onto B\onto B_r,$$
with $B_u:=(B\rtimes_\delta\dualG)_u^G$ and $B_r:=(B\rtimes_\delta\dualG)_r^G$, respectively,
are given by the identity  map on $B_c$ and it is
 easy to check that these maps are $\dualG$-equivariant
with respect to the coactions $\delta_u, \delta$, and $\delta_r$, respectively (use $(B,\delta)\cong (B_\mu,\delta_\mu)$
for a suitable crossed-product norm $\|\cdot\|_\mu$ and formula (\ref{eq-coaction})).
This gives a very concrete picture for the  maximalization $(B_u,\delta_u)$ and the normalization
$(B_r,\delta_r)$
of $(B,\delta)$ and their connections to the given coaction $(B,\delta)$.
\end{remark}

Suppose now that $H=N$ is a closed {\em normal} subgroup of $G$, and $\delta$ is an arbitrary coaction of $G$ on $B$.
Consider the representation $j_B\rtimes j_G|:B\rtimes_{\delta|}\widehat{G/N}\to\M(B\rtimes_\delta\dualG)$.
It is clear that it maps the dense subset $i_B(B_c)i_{G/N}(C_c(G/N))$ of
$B\rtimes_{\delta|}\widehat{G/N}$ onto the dense subspace $j_B(B_c)j_G(C_c(G/N))$ of $\D_N$ (where $(i_B, i_{G/N})$
denotes the canonical covariant representation of $(B,\delta|)$ into $\M(B\rtimes_{\delta|}\widehat{G/N})$), which implies
that $j_B\rtimes j_G|$ maps $B\rtimes_{\delta|}\widehat{G/N}$ onto the closure $(B\rtimes_\delta\dualG)_r^N$ of $D_N$ in $\M(B\rtimes_\delta\dualG)$, that is,
$$\im(j_B\rtimes j_G|)\cong (B\rtimes_\delta\dualG)^N_\red.$$
Moreover, by Lemma~3.1 in \cite{Kaliszewski-Quigg:Imprimitivity}, $j_B\rtimes j_G|$ is faithful if and only if
$\ker(j_B)\sbe \ker(i_B)$ (this also follows from our Lemma~\ref{lem:injectivityCovRep}). In particular, if $\delta$ is a normal coaction, \ie, if
$j_B:B\to \M(B\rtimes_\delta\dualG)$ is injective, then $j_B\rtimes j_G|$ can be viewed as an isomorphism:
$$B\rtimes_{\delta|}\widehat{G/N}\congto (B\rtimes_\delta\dualG)^N_\red.$$
This gives an alternative (and probably more concrete) description of the
isomorphism  $B\rtimes_{\delta|}\widehat{G/N}\cong (B\rtimes_\delta\dualG)_r^N$ of Theorem~\ref{thm:GPFA(BxG)N}
for normal coactions.

In case of maximal coactions $\delta=\delta_u$, there are canonical \Star{}homomorphisms
$$l_B: B\to \M((B\rtimes_\delta\dualG)_u^N)\quad\text{and}\quad
l_{G/N}: C_0(G/N)\to  \M((B\rtimes_\delta\dualG)_u^N)$$
in which for $b\in B_c$, the element $l_B(b)$ acts on the dense subalgebra $\D_N$ via multiplication
inside $\M(B\rtimes_\delta\dualG)$ (see equation (\ref{eq-multiply})). Similarly,
 if $f\in C_c(G/N)$, then $l_{G/N}(f)$ is  determined via the  obvious left and right actions
of $j_G(f)$ on $\D_N$. The following corollary is then a straightforward consequence
 of Theorem~\ref{thm:GPFA(BxG)N}:

\begin{corollary}\label{cor-max}
Let $(B,\delta)$ be a maximal coaction of $G$ and let $N$ be a normal subgroup of $G$. Then
there is a unique covariant homomorphism $(l_B,l_{G/N})$ of $(B,\delta|)$ into
 $\M((B\rtimes_\delta\dualG)_u^N)$  given on $B_c$ and $C_c(G/N)$ as above such that
 the integrated form $l_B\rtimes l_{G/N}$ implements  the isomorphism
$$l_B\rtimes l_{G/N}:B\rtimes_{\delta|}\widehat{G/N}\to (B\rtimes_\delta\dualG)_u^N$$
of Theorem~\ref{thm:GPFA(BxG)N}.
\end{corollary}
\begin{proof} In the notation of Theorem~\ref{thm:GPFA(BxG)N}, we have $B=B_u$ and $A=B\rtimes_\delta\dualG$.
If we follow the arguments given in the proof of that theorem
we see that the isomorphism $B\rtimes_{\delta|}\dual{G/N}\cong (B\rtimes_\delta\dualG)_u^N$
is given on the level of $(B\rtimes_\delta\dualG)_c^N\subseteq \M(B\rtimes_\delta\dualG)$
by sending $B_c\subseteq (B\rtimes_\delta\dualG)_c^G$ to $j_B(B_c) \subseteq\M(B\rtimes_\delta\dualG)$ and $C_c(G/N)$
to $j_G|(C_c(G/N))\subseteq \M(B\rtimes_\delta\dualG)$. But this means that
it coincides on these dense subspaces of $B$ and $C_0(G/N)$ with $l_B$ and $l_{G/N}$,
which implies the result.
\end{proof}

Of course, the results presented in this section also provide versions of Mansfield's imprimitivity theorems
for "crossed products by homogeneous spaces" as considered  in
\cite{Echterhoff-Kaliszewski-Raeburn:Crossed_products_dual, anHuef-Raeburn:Mansfield}:
if $\delta$ is a  coaction of $G$ and $H$ is a closed
subgroup of $G$, then the {\em reduced crossed product} $B\rtimes_{\delta,r}\dual{G/H}$ of
$B$ by the homogeneous space $G/H$ is defined in \cite{Echterhoff-Kaliszewski-Raeburn:Crossed_products_dual}
as
$$B\rtimes_{\delta,r}\dual{G/H}=\overline{j_B(B)j_G(C_0(G/H))}\subseteq \M(B\rtimes_{\delta}\dualG).$$
Since the reduced fixed-point algebra is the closure of $(B\rtimes_\delta\dualG)_c^H$ inside
$\M(B\rtimes_{\delta}\dualG)$, which coincides with the closure of $\D_H$ by
Lemma \ref{lem-compact supports}, we see that
$$B\rtimes_{\delta,r}\dual{G/H}\cong (B\rtimes_\delta\dualG)_r^H$$
and our results provide us with an imprimitivity bimodule $\F_r^H(B\rtimes_\delta\dualG)$
between $B\rtimes_{\delta,r}\dual{G/H}$ and $(B\rtimes_\delta\dualG)\rtimes_rH$.
This coincides with the one obtained in \cite{ anHuef-Raeburn:Mansfield}.
On the other extreme, it makes perfect sense to {\em define} the {\em universal crossed product $B\rtimes_{\delta, u}\widehat{G/H}$}
of $B$ by the homogeneous space $G/H$ as the universal fixed-point algebra
\begin{equation}\label{eq:DefRestrictedCPG/H}
B\rtimes_{\delta, u}\widehat{G/H}:=(B\rtimes_\delta\dualG)_u^H,
\end{equation}
which then provides us with the Morita equivalence $\F_u^H(B\rtimes_{\delta}\dualG)$ between
$B\rtimes_{\delta, u}\widehat{G/H}$ and $(B\rtimes_\delta\dualG)\rtimes_uH$.
Note that for normal $N$ and arbitrary coactions $\delta$ we get isomorphisms
$$B\rtimes_{\delta,r}\dual{G/N}\cong B_r\rtimes_{\delta_r|}\widehat{G/N}\quad\text{and}\quad
B\rtimes_{\delta, u}\widehat{G/N}\cong B_u\rtimes_{\delta_u|}\widehat{G/N}$$
but it is important here to use the normalization $(B_r,\delta_r)$ in the first isomorphism and the maximalization
$(B_u,\delta_u)$ in the second, since in general we do not have isomorphisms between
$B_u\rtimes_{\delta_u|}\widehat{G/N}$ and $B_r\rtimes_{\delta_r|}\widehat{G/N}$ (\eg, take
$N=G$, in which we obtain the algebras $B_u$ and $B_r$, which are often different if $G$ is not
amenable).

As far as we know, there was  no general definition of the universal crossed product by homogeneous spaces as in~\eqref{eq:DefRestrictedCPG/H} before.
However, such crossed products have been defined in the special case of dual coactions,
\ie, in the case $B=A\rtimes_\alpha G$
with  dual coaction $\delta=\dual\alpha$ for some $G$-algebra $(A,\alpha)$.
In this situation the crossed product $B\rtimes_{\delta,u}\dual{G/H}$ has been defined
in \cite{Echterhoff-Kaliszewski-Raeburn:Crossed_products_dual}
as the crossed product
$C_0(G/H,A)\rtimes_{\tau\otimes\alpha}G$, where here $\tau$ denotes left translation of $G$ on $G/H$
(see the discussion before Lemma~2.4 in \cite{Echterhoff-Kaliszewski-Raeburn:Crossed_products_dual}).
Let us now check  that both definitions agree in this case.
By the Imai-Takai Duality Theorem (\cite{Echterhoff-Kaliszewski-Quigg-Raeburn:Categorical}*{Theorem~A.67}), we have a canonical isomorphism
of weak $G\rtimes G$\nb-algebras:
$$B\rtimes_\delta\dualG\cong A\otimes \K(L^2G)$$
where $A\otimes\K(L^2G)$ is endowed with the $G$-action $\alpha\otimes\Ad_\rho$ and the structure map $1\otimes M\colon \contz(G)\to \M(A\otimes\K(L^2G))$.
But then Proposition~5.7 in \cite{Buss-Echterhoff:Imprimitivity} shows that
$$(B\rtimes_\delta\dualG)^H_\un\cong (A\rtimes\K(L^2G))^H_\un\cong \Ind_H^G(A)\rtimes_{\Ind_{\alpha}}G\cong \contz(G/H,A)\rtimes_{\tau\otimes\alpha} G.$$

\section{Twisted Landstad Duality}
\label{sec:Twisted-Landstad}

Let $G$ be a locally compact group and $N$ a closed normal subgroup of $G$. In this section we are going to study weak $G\rtimes N$-algebras, that is,
\cstar{}algebras $A$ endowed with an $N$-action $\alpha$ of $N$ and an $N$-equivariant nondegenerate \Star{}homomorphism $\phi\colon \contz(G)\to \M(A)$, where
$\contz(G)$ is endowed with right translation action of $N$. Since this action is free and proper, the corresponding Hilbert $A\rtimes_{\alpha,\pn}N$-module $\F_\pn(A)$ implements a Morita equivalence between $A^N_\pn\cong \K(\F_\pn(A))$ and $A\rtimes_{\alpha,\pn}N$ for every crossed-product norm $\|\cdot \|_\pn$ on $\contc(N,A)$. Assume now that $\|\cdot\|_\pn$ is a norm for which the dual $N$-coaction $\dualalpha$ on $A\rtimes_{\alpha,\un}N$ factors through a coaction $\dualalpha_\pn$ on $A\rtimes_{\alpha,\pn}N$.
By Lemma~4.12 in \cite{Buss-Echterhoff:Exotic_GFPA}, we know that $\F^N_\pn(A)$ carries a $G$-coaction $\delta_\F$ given by:
\begin{equation}\label{eq:FormulaCoationOnF(A)}
\delta_\F(\xi)=(\phi\otimes\Id)(w_G)(\xi\otimes 1)\quad\mbox{for all }\xi\in \F_c(A)=\phi(\contc(G))A
\end{equation}
which implements a Morita equivalence between the inflation $\Inf\dualalpha_\pn$ of the dual coaction $\dualalpha_\pn$ on $A\rtimes_{\alpha,\pn} N$, and the $G$-coaction $\delta^N_\pn$ on $A^N_\pn$ induced by $\delta_\F$ which is given by:
\begin{equation}\label{eq-deltap}
\delta^N_\pn(m)=(\phi\otimes\id)(\omega_G)(m\otimes 1)(\phi\otimes\id)(\omega_G)^*\quad\mbox{for all }m\in A^N_c.
\end{equation}
(see \cite[Remark~4.14 ]{Buss-Echterhoff:Exotic_GFPA}).
Our first goal is to show that $\delta_\pn$ is a twisted coaction in the following sense:

\begin{definition}\label{def:TwistedCoaction}
A twisted coaction of $(G,G/N)$ on a \cstar{}algebra $B$ is a pair $(\delta,\twu)$ consisting of a (nondegenerate) $G$-coaction $\delta\colon B\to \M(B\otimes \Cst(G))$ of $G$ on $B$
and a \emph{twisting unitary} over $G/N$, meaning a unitary multiplier $\twu\in \U\M(B\otimes \Cst(G/N))$ satisfying:
\begin{enumerate}[(i)]
\item $(\twu\otimes 1)(\Id\otimes\sigma_{G/N,G/N})(\twu\otimes 1)=(\Id\otimes \delta_{G/N})(\twu)$;
\item $(\delta\otimes\Id_{G/N})(\twu)=(\Id_A\otimes\sigma_{G/N,G})(\twu\otimes 1)$; and
\item $\delta|(b)=\twu(b\otimes 1)\twu^*$ for all $b\in B$,
\end{enumerate}
where $\sigma_{G/N,G/N}$ and $\sigma_{G/N,G}$ denote the flip isomorphisms on $\Cst(G/N)\otimes \Cst(G/N)$ and $\Cst(G/N)\otimes \Cst(G)$, respectively,
and $\delta|\defeq (\Id\otimes q_N)\circ\delta$ denotes the restriction of $\delta$ to $G/N$, where $q_N\colon \Cst(G)\to \Cst(G/N)$ is the quotient map.

Equivalently (see \cite{Quigg-Raeburn:Induced}), a twisted coaction can be defined
 as a pair $(\delta,\twh)$ consisting of a coaction $\delta:B\to \M(B\otimes C^*(G))$
 and a nondegenerate \Star{}homomorphism $\twh\colon \contz(G/N)\to \M(B)$ satisfying:
\begin{enumerate}[(1)]
\item $(\iota,\twh)$ is a covariant representation of $(B,\delta|)$ into $\M(B)$, where
$\iota\colon B\to \M(B)$ denotes the inclusion map, that is,
$$\delta|(b)=(\twh\otimes\id)(\omega_{G/N})(b\otimes 1)(\twh\otimes \id)(\omega_{G/N}^*)\quad\mbox{for all }b\in B;\mbox{ and }$$
\item $\delta(\twh(f))=\twh(f)\otimes 1$ for all $f\in C_0(G/N)$.
\end{enumerate}
In this case, $\twh$ is called the \emph{twisting homomorphism} for $(B,\delta)$.
\end{definition}

If the twisting homomorphism $\twh$ is given, the unitary twist $\twu$ can be recovered from $\twh$ by $\twu=(\twh\otimes\id)(\omega_{G/N})$. Conversely, every unitary twist is of this form by \cite[Lemma~A.1]{Quigg-Raeburn:Induced}, and in this case we say that $\twh$ is the twisting homomorphism associated to $\twu$ or that $\twu$ is the
twisting unitary associated to $\twh$.
We refer to  \cite{Quigg-Raeburn:Induced,Echterhoff-Raeburn:The_stabilisation_trick,Phillips.Raeburn:Twisted} for further information on twisted coactions.

To simplify the writing, we shall use standard leg numbering notations like $\twu_{12}\defeq \twu\otimes 1$, $\twu_{23}\defeq 1\otimes \twu$ and $\twu_{13}=(\Id\otimes \sigma)(\twu)$, where $\sigma$ is some suitable flip automorphism (like $\sigma_{G/N,G/N}$ or $\sigma_{G/N,G}$ as above). With these notations, the two first conditions in the above definition can be  reformulated as:
$$\text{(i) $\twu_{12}\twu_{13}=(\Id\otimes \delta_{G/N})(\twu)$}\quad\quad \text{and}\quad\quad \text{(ii) $(\delta\otimes\Id_{G/N})(\twu)=\twu_{13}$}.$$
The first condition can be interpreted by saying that $\twu\in\U\M(B\otimes \Cst(G/N))$ is a corepresentation of $G/N$ on $B$. Observe that, in this case, if $\psi\colon B\to \M(C)$ is a \Star{}homomorphism, then $(\psi\otimes\id)(\twu)$ is a corepresentation of $G/N$ on $C$.

\begin{definition}
Let $(B,\delta,\twu)$ be a twisted coaction of $(G,G/N)$. We say that a covariant representation $(\pi,\sigma)$ of $(B,\delta)$ \emph{preserves the twist} if
\begin{equation}\label{eq-twist}
(\sigma\otimes\Id_{G/N})(\omega_{G/N})=(\pi\otimes\Id_{G/N})(\twu).
\end{equation}
 In this case we also say that $(\pi,\sigma)$ is a covariant representation of $(B,\delta, \twu)$.

A \emph{twisted crossed product} for $(B,\delta,\twu)$ is a \cstar{}algebra $C$ endowed with a covariant representation $(k_B,k_G)$
of $(B,G,\twu)$ into $\M(C)$
 such that $k_B(B)k_G(\contz(G))$ is linearly dense in $C$ and such that for every other twisted covariant representation $(\pi,\sigma)$ into
 $\M(D)$ there
 exists a unique nondegenerate representation $\pi\rtimes_\twu\sigma: C\to \M(D)$ such that $\pi=(\pi\rtimes_\twu\sigma)\circ k_B$ and
 $\sigma=(\pi\rtimes_\twu\sigma)\circ k_G$.
\end{definition}

\begin{remark}\label{rem:PreservationOfTwist}
If $\twh\colon \contz(G/N)\to \M(B)$ is the twisting homomorphism associated to $\twu$, then a covariant representation $(\pi,\sigma)$ preserves the twist if and only if
$\sigma|_{\contz(G/N)}=\pi\circ \twh$ (see \cite{Phillips.Raeburn:Twisted}*{Remark~2.6}).
\end{remark}

Every twisted coaction admits a twisted crossed product which is uniquely determined up to isomorphism and denoted by $B\rtimes_{\delta,\twu}\dualG$.
If $B\rtimes_\delta \dualG$ denotes the (untwisted) crossed product for $(B,\delta)$ and $(j_B,j_G)$ is its universal covariant representation, then $B\rtimes_{\delta,\twu}\dualG$ can be realized as the quotient $B\rtimes_{\delta,\twu}\dualG=(B\rtimes_\delta\dualG)/{I_\twu}$, where
$$I_\twu=\cap\{\ker(\pi\rtimes \sigma):(\pi,\sigma) \mbox{ is covariant representation preserving the twist}\}$$
is the \emph{twisting ideal}. In this picture, the universal covariant representation $(k_B,k_G)$ is just the composition $(q_\twu\circ j_B,q_\twu\circ j_G)$, where $q_\twu\colon B\rtimes_\delta \dualG\to B\rtimes_{\delta}\dualG/I_\twu$ is the quotient map and $\pi\rtimes_\twu\sigma$ is the unique factorization of $\pi\rtimes\sigma$ through $(B\rtimes_\delta\dualG)/{I_\twu}$.

\begin{example}\label{ex:twisted_coactions}
(1) The inflation $\Inf\delta$ of an $N$-coaction $\delta\colon B\to \M(B\otimes \Cst(N))$ is a \emph{trivially twisted coaction} over $G/N$, that is, a twisted coaction of $(G,G/N)$ with respect to the trivial unitary twist $\twu=1$ (this corresponds to the \emph{trivial twisting homomorphism} $\twh\colon\contz(G/N)\to\M(B)$ defined by $\twh(f)=f(eN)1_B$ for all $f\in\contz(G/N)$). The twisted crossed product $B\rtimes_{\Inf\delta,1}\dualG$ is canonically isomorphic to the original (untwisted) crossed product $B\rtimes_\delta\dual{N}$ (see \cite{Phillips.Raeburn:Twisted, Echterhoff-Raeburn:The_stabilisation_trick}).

(2) Given an arbitrary coaction $\delta\colon B\to \M(B\otimes \Cst(G))$ of $G$ on $B$, the restricted crossed product $B\rtimes_{\delta|}\dualGN$ carries a canonical twisted $(G,G/N)$-coaction $(\tilde\delta,\tilde\twu)$: the coaction $\tilde\delta$ is the integrated form $\pi\rtimes\sigma$ of the covariant representation
$$(\pi,\sigma)=((j_B\otimes\Id)\circ\delta,j_{G/N}\otimes 1),$$
where $(j_B,j_{G/N})$ denotes the universal covariant representation of $(B,\delta|)$
and the twisting homomorphism for $\tilde\delta$ is $j_{G/N}$, hence $\tilde\twu=(j_{G/N}\otimes \id)(\twu_{G/N})$.
Moreover, there is a canonical isomorphism (see \cite{Phillips.Raeburn:Twisted} and also Remark~7.12 in \cite{Quigg-Raeburn:Induced}):
$$(B\rtimes_{\delta|}\dualGN)\rtimes_{\tilde\delta,\tilde\twu}\dualG\cong B\rtimes_{\delta}\dualG.$$
In  Corollary~\ref{cor:PR-DecompositionTheo} below we derive this decomposition isomorphism also as a consequence of our results.
\end{example}

If $\twu$ a twisting unitary over $G/N$ for $(B,\delta)$, then the twisting ideal $I_\twu$ is $N$-invariant with respect to the dual action $\widehat{\delta}$, so that
$\widehat{\delta}$ induces an $N$-action on the twisted crossed product $B\rtimes_{\delta,\twu}\dualG$, which we  denote by $\dual\delta^\twu$. If $(k_B,k_G)$ denotes the
universal twisted covariant representation, the homomorphism $k_G\colon\contz(G)\to \M(B\rtimes_{\delta,\twu}\dualG)$ is $N$-equivariant with respect
to the right translation action of $N$ on $G$, and hence $B\rtimes_{\delta,\twu}\dualG$
carries a canonical structure as a weakly proper $G\rtimes N$-algebra.

The following result is well-known. In \cite[Theorem~4.4]{Quigg-Raeburn:Induced}  it is shown for normal \emph{amenable} subgroups and \emph{reduced coactions}, that is, injective nondegenerate coactions of $\Cst_\red(G)$. But it is pointed out in the proof of \cite[Theorem~4.3]{Kaliszewski-Quigg:Imprimitivity} that the  proof of \cite[Theorem~4.4]{Quigg-Raeburn:Induced} extends to arbitrary (\ie, also non-amenable) normal subgroups if one replaces reduced coactions by  \emph{full} coactions of $C^*(G)$ as we are using here.

\begin{proposition}[Quigg-Raeburn]
\label{QR-Cocrossedproduct=IndAlg}
For a twisted $(G,G/N)$-coaction $(B,\delta,\twu)$, there is a $G$-equivariant isomorphism
$\chi\colon B\rtimes_\delta\dualG\congto \Ind_N^G(B\rtimes_{\delta,\twu}\dualG)$ sending $x\in B\rtimes_\delta\dualG$ to the function
$t\mapsto (k_B\rtimes k_G)(\dual\delta_{t^{-1}}(x))\in \Ind_N^G(B\rtimes_{\delta,\twu}\dualG)$.\end{proposition}

\begin{remark}\label{rem-chi-iso} It is clear that the extension of $\chi$ to the multiplier algebra $\M(B\rtimes_\delta\dualG)$  sends
$j_B(b)$ to the constant function $\tilde{k}_B(b)=(t\mapsto k_B(b))\in \M(\Ind_N^G(B\rtimes_{\delta,\twu}\dualG))$ and $j_G(f)$ to the element
$\tilde{k}_G(f)\in \M( \Ind_N^G(B\rtimes_{\delta,\twu}\dualG))$ determined
by the function $t\mapsto k_G(\tau_{t^{-1}}(f))\in \M(B\rtimes_{\delta,\twu}\dualG)$ with $\tau_t(f)|_s=f(st)$.
Hence,  $\chi$ becomes an isomorphism between the
weak $G\rtimes G$-algebras $(B\rtimes_\delta\dualG, j_G, \widehat\delta)$ and $(\Ind_N^G(B\rtimes_{\delta,\twu}\dualG), \tilde{k}_G, \Ind\dual\delta^\twu)$.

 Observe that the above proposition implies, in particular, that there is a central homomorphism $\psi\colon\contz(G/N)\to \Zt\M(B\rtimes_\delta\dualG)$ corresponding to the canonical homomorphism $f\mapsto f\otimes 1$ from $\contz(G/N)$ into $\Zt\M\big(\Ind_N^G(B\rtimes_{\delta,\twu}\dualG)\big)$. A formula for $\psi$ is given in \cite{Quigg-Raeburn:Induced}*{Theorem~4.4}: it is a certain convolution of $j_G|_{\contz(G/N)}$ and $j_B\circ \twh$, where $\twh$ is the twisting homomorphism $\contz(G/N)\to \M(B)$ associated to $\twu$.
\end{remark}

An easy consequence of the above proposition is the following:

\begin{corollary}
For a twisted coaction $(B,\delta,\twu)$, the (untwisted) coaction $(B,\delta)$ is normal if and only if $k_B$ is injective.
Moreover, we have $\ker(j_B)=\ker(k_B)$.
\end{corollary}
\begin{proof}
Recall that $\delta$ is normal if and only if $j_B\colon B\to \M(B\rtimes_\delta\dualG)$ is injective. Thus it is enough to prove the final assertion.
Let $q\defeq k_B\rtimes k_G\colon B\rtimes_\delta\dualG\to B\rtimes_{\delta,\twu}\dualG$ denote the canonical surjection and let $\chi\colon B\rtimes_\delta\dualG\congto \Ind_N^G(B\rtimes_{\delta,\twu}\dualG)$ be the isomorphism of Proposition~\ref{QR-Cocrossedproduct=IndAlg}. Then, for all $b\in B$, since $j_B(b)$ is $\dual\delta$-invariant, we have
$$\chi(j_B(b))|_t=q(\dual\delta_{t^{-1}}(j_B(b)))=q(j_B(b))=k_B(b).$$
This implies $\ker(j_B)=\ker(k_B)$ because $\chi$ is injective.
\end{proof}

The following result appears as Corollary~4.10 in \cite{Quigg-Raeburn:Induced} where, again, amenability of $N$ is required due to the use of reduced coactions. However, using Proposition~\ref{QR-Cocrossedproduct=IndAlg} above, the same proof as given  in \cite{Quigg-Raeburn:Induced} applies to full normal coactions and non-amenable $N$.

\begin{lemma}[Quigg-Raeburn]
\label{lem:injectivityCovRep}
Let $(\pi,\sigma)$ be a covariant representation of a twisted $(G,G/N)$-coaction $(B,\delta,\twu)$ into $\M(D)$. Assume that $(B,\delta)$ is normal.
Then $\pi\rtimes_\twu\sigma$ is faithful if and only if $\pi$ is faithful and there is an action of $N$ on the image of $\pi\rtimes_\twu\sigma$ making $\sigma$ into an $N$-equivariant homomorphism.
(The last condition is equivalent to saying that $\ker(\pi\rtimes_\twu\sigma)$ is an $N$-invariant ideal in $B\rtimes_{\delta,\twu}\dualG$.)
\end{lemma}

Suppose that $(B,\delta,\twu)$ is a twisted $(G,G/N)$-coaction and let $(B_\red,\delta_\red)$ be the normalization of $(B,\delta)$. Then
it is well-known that $B\rtimes_\delta\dualG\cong B_\red\rtimes_{\delta_\red}\dualG$. In particular,
covariant representations of $(B,\delta)$ correspond bijectively to covariant representations of $(B_\red,\delta_\red)$.
This correspondence can be described as follows: recall that $(B_\red,\delta_\red)$ can be realized as $B_\red=j_B(B)\cong B/\ker{j_B}$ and $\delta_\red$ is given
on $j_B(B)$ by conjugation with the unitary $(j_G\otimes\id)(\omega_G)$.
If $(\pi,\sigma)$ is a covariant representation of $(B,\delta)$, the equation $(\pi\rtimes\sigma)\circ j_B=\pi$ implies that $\ker(j_B)\sbe \ker(\pi)$, so that $\pi$
factors through a homomorphism $\pi_\red$ of $B_\red$ and the pair $(\pi_\red,\sigma)$ is a covariant representation of $(B_\red,\delta_\red)$.
The assignment $(\pi,\sigma)\mapsto (\pi_\red,\sigma)$ is then a bijective correspondence between covariant representations of $(B,\delta)$ and $(B_\red,\delta_\red)$.

Moreover, if $\varrho\colon B\onto B_\red$ denotes the quotient map, the unitary twist $\twu$ for $\delta$ induces a twist $\twu_\red:=(\varrho\otimes \id)(\twu)$ for $\delta_\red$,
and a covariant representation $(\pi,\sigma)$ of $(B,\delta)$ preserves the twist $\twu$ if and only if $(\pi_\red,\sigma)$ preserves the twist $\twu_\red$.
It follows that the canonical surjection $\varrho\colon B\to B_\red$ induces an isomorphism of weak $G\rtimes N$-algebras:
\begin{equation}\label{eq:IsomorphismTwistedCrossedProducts=Normalization}
\varrho\rtimes_\twu\dualG\colon B\rtimes_{\delta,\twu}\dualG\congto B_\red\rtimes_{\delta_\red,\twu_\red}\dualG.
\end{equation}
Using this observation, we obtain the following generalization of Lemma \ref{lem:injectivityCovRep}  to arbitrary twisted coactions:

\begin{lemma}\label{lem:gen-injectivityCovRep}
Let  $(\pi,\sigma)$ be a covariant representation of the $(G,G/N)$-twisted coaction $(B,\delta,\twu)$
such that there exists an action of $N$ on the image of $\pi\rtimes_\twu\sigma$ making $\sigma$ into an $N$-equivariant homomorphism.
Then $\pi\rtimes_\twu\sigma$ is faithful if and only if $\ker\pi= \ker k_B\; (=\ker\big(\varrho\colon B\onto B_\red))$.
\end{lemma}

Recall that two twisted coactions $(B,\delta_B,\twu_B)$ and $(C,\delta_C,\twu_C)$ are \emph{Morita equivalent} if there is an imprimitivity $A$--$B$-bimodule
 $\E$ carrying a $G$-coaction $\delta_\E$ compatible with $\delta_A$ and $\delta_B$ and satisfying $\delta_\E|(x)\defeq (\Id\otimes q_N)\circ\delta_\E(x)=\twu_A(x\otimes 1)\twu_B^*$ for all
 $x\in \E$.
The following result is a twisted version of the Landstad Duality Theorem for coactions we proved in \cite{Buss-Echterhoff:Exotic_GFPA} (which is, in turn, a generalization of the main result in \cite{Quigg:Landstad_duality}).

\begin{theorem}\label{theo:LandstadTwistedTheorem1}
Let $(A,\alpha,\phi)$ be a weak $G\rtimes N$-algebra and let $\|\cdot\|_\pn$ be a crossed-product norm on $\contc(N,A)$ for which the dual $N$-coaction $\dualalpha$ on $A\rtimes_{\alpha,\un}N$ factors through a coaction $\dualalpha_\pn$ on $A\rtimes_{\alpha,\pn}N$. Then
\begin{enumerate}[(i)]
\item the $G$-coaction $\delta^N_\pn$ on $A^N_\pn$ given by Equation~\eqref{eq-deltap} is twisted over $G/N$ with twisting homomorphism $\phi^N\colon \contz(G/N)\to \M(A^N_\pn)$ induced from the structural homomorphism $\phi\colon \contz(G)\to \M(A)$ as in (\ref{eq-leftaction}), \ie,
    $\omega_\pn\defeq (\phi^N\otimes\id)(\omega_{G/N})$ is the twisting unitary for $\delta^N_\pn$.
\item The coaction $\delta_\F$ on $\F=\F_\pn^N(A)$ implements a Morita equivalence between $(A^N_\pn,\delta^N_\pn,\twu_\pn)$ and the trivially twisted coaction $(A\rtimes_{\alpha,\pn}N,\Inf\dualalpha_\pn,1)$.
Moreover, $(A^N_\un,\delta^N_\un)$ is a maximal $G$-coaction and $(A^N_\red,\delta^N_\red)$ is a normal $G$-coaction.
\item If $\kappa\colon A^N_\pn\to \M(A)$ is the canonical representation given by the extension of the inclusion map $A^N_c\into \M(A)$
\textup{(}see \cite{Buss-Echterhoff:Exotic_GFPA}*{Proposition~3.5}\textup{)},
then the pair $(\kappa,\phi)$ is a covariant representation of $(A^N_\pn,\delta^N_\pn,\twu_\pn)$ into $\M(A)$
and the corresponding integrated form $\kappa\rtimes_\twu\phi$ is an isomorphism $A^N_\pn\rtimes_{\delta^N_\pn,\twu_\pn}\dualG\congto A$
of weak $G\rtimes N$-algebras.
\end{enumerate}
\end{theorem}
\begin{proof}
To prove (i), we have to verify the conditions in (1) and (2) in Definition~\ref{def:TwistedCoaction} for the homomorphism $\phi^N$ and the coaction $\delta^N_\pn$.
The  condition $\delta^N_\pn\circ\phi^N=\phi^N\otimes 1$ follows from Equation~\eqref{eq-deltap} and the relation
$(f\otimes 1)\omega_G=\omega_G(f\otimes 1)$ for all $f\in\contz(G)$ (remember that $\omega_G\in \M(\contz(G)\otimes \Cst(G))$).

In order to prove the  condition $\delta^N_\pn|(m)=\twu_\pn(m\otimes 1)\twu_\pn^*$ for all $m\in A^N_c$
we choose $f\in C_c(G/N)$ such that $m=\phi(f)m$.
We then compute, for $z\in \Cst(G)$,
$$(\phi\otimes q_N)(\omega_G)(m\otimes q_N(z))=(\phi\otimes \id)(\id\otimes q_N)(\omega_G(f\otimes z))(m\otimes 1)$$
and also
$$(\phi^N\otimes \id)(\omega_{G/N})(m\otimes q_N(z))=(\phi^N\otimes \id)(\omega_{G/N}(f\otimes q_N(z)))(m\otimes 1).$$
Now observe that $(\id\otimes q_N)(\omega_G(f\otimes z))$ is the function in $\contb(G, \Cst(G/N))$
 given by $s\mapsto f(sN)u_{sN}q_N(z)$, which is constant on $N$-cosets and factors to the function  $\omega_{G/N}(f\otimes q_N(z))$
 in $\contc(G/N, C^*(G/N))$.
  Since $\phi\otimes \id$ restricts to $\phi^N\otimes\id$ on
 this space, we conclude that $(\phi\otimes q_N)(\omega_G)(m\otimes q_N(z))=(\phi^N\otimes \id)(\omega_{G/N})(m\otimes q_N(z))$ for all
 $m\in A_c^N$ and $z\in C^*(G)$, which then implies that $(\phi\otimes q_N)(\omega_G)=(\phi^N\otimes \id)(\omega_{G/N})$ in  $\M(A^N_\pn\otimes \Cst(G/N))$.
 Therefore
\begin{multline*}
\delta^N_\pn|(m)=(\phi\otimes q_N)(\omega_G)(m\otimes 1)(\phi\otimes  q_N)(\omega^*)\\
=(\phi^N\otimes \id)(\omega_{G/N})(m\otimes 1)(\phi^N\otimes  \id)(\omega_{G/N}^*)=\twu_\pn(m\otimes 1)\twu_\pn^*.
\end{multline*}
Therefore $(\delta^N_\pn,\twu_\pn)$ is a twisted action of $(G,G/N)$ on $A^N_\pn$.
Moreover, the same argument just used to verify axiom (1) in Definition~\ref{def:TwistedCoaction} yields, for all $\xi\in \F_c(A)$,
$$\delta_\F|(\xi)=(\id\otimes q_N)\delta_\F(\xi)=(\phi\otimes q_N)(\omega_G)(\xi\otimes 1)=(\phi^N\otimes \id)(\omega_{G/N})(\xi\otimes 1)=\twu_\pn(\xi\otimes 1)$$
which is saying that $(\F,\delta_\F)$ implements the desired Morita equivalence between $(A^N_\pn,\delta^N_\pn,\twu_\pn)$ and $(A\rtimes_{\alpha,\pn}N,\Inf\dualalpha,1)$.
This proves the first assertion in (ii) and the second assertion follows from the fact that maximality and normality of coactions  are preserved by Morita equivalence and by inflation of coactions
 (see \cite{Echterhoff-Kaliszewski-Quigg:Maximal_Coactions}*{Proposition~3.5}, \cite{Kaliszewski-Quigg:Mansfield}*{Proposition~7.3} and
 \cite{Echterhoff-Kaliszewski-Quigg-Raeburn:Categorical}*{Lemma~3.19}).

Finally, to prove (iii) we first observe that $(\kappa,\phi)$ is a covariant representation of $(A^N_\pn,\delta^N_\pn)$, that is, that
$$(\kappa\otimes\id)(\delta^N_\pn(a))=(\phi\otimes\id)(\omega_G)(\kappa(a)\otimes 1)(\phi\otimes\id)(\omega_G^*)$$
 for all $a\in A^G_\pn$. Of course, it suffices to verify this equation for $a\in A^G_c$ and then it follows directly from formula~\eqref{eq-deltap}. Therefore $(\kappa,\phi)$ is a covariant representation of $(A^N_\pn,\delta^N_\pn)$ into $\M(A)$ and an argument similar to that given in the proof of \cite[Lemma~3.10(2)]{Quigg:Landstad_duality} (replacing $G$ by $N$ where appropriate)
shows that the image of $\kappa\rtimes\phi$ is $A$, so that we may view $\kappa\rtimes\phi$ as a surjective \Star{}homomorphism from $A^N_\pn\rtimes_{\delta^N_\pn}\dualG$ onto $A$.
Moreover, it is easy to see that $\kappa\rtimes\phi$ commutes with the $N$- and $\contz(G)$-actions.
Since $\phi|_{\contz(G/N)}=\kappa\circ\phi^N$, it follows from Remark~\ref{rem:PreservationOfTwist} that the covariant representation $(\kappa,\phi)$ preserves the twist
$\twu_\pn=\phi^N\otimes\id(\omega_{G/N})$.
We  need to show that the  $G\rtimes N$-equivariant \Star{}homomorphism
$$\kappa\rtimes_{\twu_\pn}\phi:A_\pn^N\times_{\delta^N_\pn,\twu_\pn}\dualG\onto A$$ is injective. But this follows from Lemma \ref{lem:gen-injectivityCovRep}
and the fact that $\kappa:A_\pn^N\to \M(A)$ factors through a faithful map $\kappa_\red:A_\red^N\to \M(A)$, hence $\ker\kappa$ coincides with the kernel
of the normalization morphism $A_\pn^N\onto A_\red^N$.
\end{proof}

In what follows next we want to show that,  conversely, every twisted coaction $(\delta,\twu)$ is  of the kind as in
Theorem~\ref{theo:LandstadTwistedTheorem1} for the weak $G\rtimes N$-algebra  $(A,\alpha,\phi)$
with
$$A=B\rtimes_{\delta,\twu}\dual G,\quad \alpha=\dual\delta^\twu, \quad\text{and}\quad \phi=k_{C_0(G/N)}.$$
In order to prepare the result, we show

\begin{lemma}\label{lem-inducedcoact}
Suppose that $(\delta,\twu)$ is a twisted coaction of $(G,G/N)$ on the \cstar{}algebra $B$.
Let $\|\cdot\|_\pn$ denote either the full crossed-product norm $\|\cdot\|_u$ or the reduced crossed-product norm
$\|\cdot\|_\red$ for crossed products by $N$ and $G$.
Then the cosystems
$$\big((B\rtimes_{\delta}\dualG)_\pn^G, \delta_\pn^G\big)\quad\text{and}\quad \big((B\rtimes_{\delta,\twu}\dualG)_\pn^N, \delta_\pn^N\big)$$
are isomorphic.
 The isomorphism maps an element $b$ of the inductive limit dense subalgebra $B_c\cong j_B(B_c)$ of
$(B\rtimes_{\delta}\dualG)_c^G$ to the element $k_B(b)$ in the inductive limit dense subalgebra $k_B(B_c)\subseteq (B\rtimes_{\delta,\twu}\dualG)_c^N$,
where $B_c=\delta_{A_c(G)}(B)\subseteq B$ (compare with Remark \ref{rem-Bc}).
In particular, we get $(B\rtimes_{\delta}\dualG)_\red^G=k_B(B)\subseteq \M(B\rtimes_{\delta,\twu}\dualG)$.
\end{lemma}
\begin{proof}
Write $A\defeq B\rtimes_{\delta,\omega}\dualG$.
By Proposition~\ref{QR-Cocrossedproduct=IndAlg}, we have $\Ind_N^G(A)\cong B\rtimes_\delta\dualG$ as weak $G\rtimes G$-algebras.
Then \cite[Theorem~4.6]{Buss-Echterhoff:Exotic_GFPA} implies that $\big((\Ind_N^G(A))^G_u,\delta^G_u\big)$ is the
maximalization and $\big((\Ind_N^G(A))^G_\red,\delta^G_\red\big)$ is the normalization of $(B,\delta)$.
The coactions  $\delta^G_\pn$ for $\pn=u,\red$ are given by the formula (as in Equation~\eqref{eq-deltap}):
\begin{equation}\label{eq:coactionInd-formula}
\delta^G_\pn(x)=(\tilde\phi\otimes\id)(\omega_G)(x\otimes 1)(\tilde\phi\otimes\id)(\omega_G^*)\quad\mbox{for all }x\in (\Ind_N^G(A))^G_c,
\end{equation}
where $\tilde\phi$ is the structural homomorphism $\contz(G)\to \M(\Ind_N^G(A))$, which is given by $\tilde\phi(f)|_t=\phi(\tau_{t^{-1}}(f))$
(with $\tau_t(f)|_s=f(st)$) acting on an element $F\in \Ind_N^GA$ by pointwise multiplication.
Now, Proposition~\ref{prop:IndPreservesFix} yields a canonical isomorphism $\psi\colon A^N_\pn\congto \big(\Ind_N^G(A)\big)^G_\pn$, which
sends $m\in A^N_c$ to the constant function $G\to \M(A)$, $t\mapsto m$, which defines an element of $(\Ind_N^G(A))^G_c\sbe \M(\Ind_N^G(A))^G$.

We now show that $\psi$ is equivariant with respect to the $G$-coactions $\delta^N_\pn$ on $A^N_\pn$
and $\delta^G_\pn$ on $(\Ind_N^G(A))^G_\pn$, that is, $\delta^G_\pn\circ\psi=(\psi\otimes\id)\circ\delta^N_\pn$, so that $\psi$ becomes an isomorphism of $\dualG$-algebras.
To prove this, recall that $\delta^N_\pn$ is given as in Equation~\ref{eq-deltap} by the formula
$$\delta^N_\pn(m)=(\phi\otimes\id)(\omega_G)(m\otimes 1)(\phi\otimes\id)(\omega_G^*)\quad\mbox{for all }m\in A^N_c.$$
As explained in \cite{Buss-Echterhoff:Exotic_GFPA}*{Remark~4.14}, the right hand side of this equation is, a priori, an element of $\M(A\otimes C^*(G))$, but can be interpreted as
an element of $(A\otimes C^*(G))^N_c\into A^N_\pn\otimes C^*(G)$. A similar interpretation is used for $\delta^G_\pn$ in~\eqref{eq:coactionInd-formula}.
Now we observe that the isomorphism $\psi\otimes \id\colon A^N_\pn\otimes C^*(G)\congto (\Ind_N^G(A))^G_\pn\otimes C^*(G)$
sends an element $x\in (A\otimes C^*(G))^N_c$ to the constant function $t\mapsto x$ from $G$ to $\M(A\otimes C^*(G))$ viewed as an element of
$\M\big((\Ind_N^G(A))^G_\pn\otimes C^*(G)\big)$. We will apply this to $x=\delta^N_\pn(m)$. We need to show that the element
$\delta^G_\pn(\psi(m))\in \M(\Ind_N^G(A)\otimes C^*(G))$ is sent via the canonical  inclusion $\M(\Ind_N^G(A)\otimes C^*(G))\into \M(\contz(G,A\otimes C^*(G)))$
to the constant function $t\mapsto \delta^N_\pn(m)$ from $G$ to $\M(A\otimes C^*(G))$. But given $t\in G$, observe that $(\tilde\phi\otimes\id)(\omega)|_t=(\phi\otimes\id)(\omega_G)(1\otimes u_{t^{-1}})$, where $t\mapsto u_t$ denotes the universal representation of $G$ into $\M(C^*(G))$ (remember that $\omega_G(s)=u_s$). Therefore,
\begin{align*}
    \delta^G_\pn(\psi(m))|_t&=(\tilde\phi\otimes\id)(\omega_G)(\psi(m)\otimes 1)(\tilde\phi\otimes\id)(\omega_G^*)|_t\\
        &=(\phi\otimes\id)(\omega_G)(1\otimes u_{t^{-1}})(\psi(m)\otimes 1)(1\otimes u_{t})(\phi\otimes\id)(\omega_G^*)\\
        &=(\phi\otimes\id)(\omega_G)(\psi(m)\otimes 1)(\phi\otimes\id)(\omega_G^*)=\delta^N_\pn(m).
\end{align*}
This proves that $\psi$ is $\dualG$-equivariant.
Finally, it follows from
Remark \ref{rem-chi-iso} and the description of the inclusion of $A_c^N$ into $(\Ind_N^G(A))^G_c$ via constant functions
that the isomorphism $A_\pn^N\cong (\Ind_N^GA)_\pn^G\cong (B\rtimes_{\delta}\dualG)_\pn^G$ maps
$k_B(B_c)$  bijectively onto $j_B(B_c)\subseteq (B\rtimes_{\delta}\dualG)_c^G$. This implies the last assertion of the lemma.
\end{proof}

\begin{remark}\label{rem-Bc1}
The above lemma together with Lemma \ref{lem-B0}  imply in particular that $k_B: B_c\to k_B(B_c)$ is an isomorphism
of \Star{}algebras. Hence we may regard $B_c$ as an inductive limit dense subalgebra of $(B\rtimes_{\delta,\twu}\dualG)_c^N$.
In particular, we see that for any crossed-product norm $\|\cdot\|_\pn$ on $C_c(N, B\rtimes_\delta\dualG)$, the
corresponding fixed-point algebra $B_\mu^N:=(B\rtimes_{\delta,\twu}\dualG)_\pn^N$ can be regarded as a completion of $B_c$
with respect to a suitable norm. Moreover, if the chosen norm $\|\cdot\|_\pn$ admits a dual coaction on
$(B\rtimes_{\delta,\twu}\dualG)\rtimes_\pn N$, we obtain a twisted coaction $(\delta_\pn^N, \twu_\pn)$ as
in Theorem~\ref{theo:LandstadTwistedTheorem1}.
\end{remark}

\begin{theorem}\label{theo:pn-twisted-coactions}
Let  $(B,\delta,\twu)$ be a twisted coaction of $(G, G/N)$. Then there exist $(G, G/N)$-equivariant epimorphisms
$$B_u^N\;\stackrel{q_u}{\onto}\; B\;\stackrel{q_{\red}}{\onto} \;B_\red^N$$
given by the identity map on $B_c$ viewed as a dense \Star{}subalgebra of all three algebras such that the
induced morphisms
$$B_u^N\rtimes_{\delta_u^N,\twu_u}\dualG\stackrel{q_u\rtimes\dualG}{\longrightarrow}
B\rtimes_{\delta,\twu}\dualG\stackrel{q_\red\rtimes\dualG}{\longrightarrow}
B_\red^N\rtimes_{\delta_\red^N,\twu_\red}\dualG$$
are isomorphisms of weakly proper $G\rtimes N$-algebras.
Moreover there exists a unique crossed-product norm $\|\cdot\|_\pn$ on $C_c(N, B\rtimes_{\delta,\twu}\dualG)$ which admits a
$(\delta_\pn^N, \twu_\pn)$--$(\delta,\twu)$ equivariant isomorphism $B_\pn^N\cong B$ extending the identity on $B_c$, and then
$\F_\pn^N(B\rtimes_{\delta,\twu}\dualG)$ induces a Morita equivalence between the twisted cosystems
$$(B,\delta,\twu)\quad\text{and}\quad \big((B\rtimes_{\delta,\twu}\dualG)\rtimes_\pn N, \Inf(\widehat{\dual\delta^\twu})_\pn, 1_{G/N}\big).$$
\end{theorem}
\begin{proof} It follows from item (ii) of Theorem~\ref{theo:LandstadTwistedTheorem1} together with
Lemma \ref{lem-inducedcoact} that $(B_u^N, \delta_u^N)$ is a maximalization of $(B,\delta)$ and $(B_\red^N,\delta_\red^N)$
is a normalization of $(B,\delta)$. Thus it follows from Remark \ref{rem-Bc} that the identity map on $B_c$ induces
$\delta_u^N, \delta, \delta_\red^N$ equivariant epimorphisms $B_u^N\;\stackrel{q_u}{\onto}\; B\;\stackrel{q_{\red}}{\onto} \;B_\red^N$.
By continuity, the composition $q_\red\circ q_u$ extends the identity map on $(B\rtimes_{\delta,\twu}\dualG)_c^G$, hence we
see that $B$ can be obtained as a completion of $B_c^N:=(B\rtimes_{\delta,\twu}\dualG)_c^G$ with respect to a suitable norm
$\|\cdot\|_\nu$. It follows then from the Rieffel-correspondence  applied to the $(G,G/N)$-equivariant
 $B_u^N$--$(B\rtimes_{\delta,\twu}\dualG)\rtimes_uN$ equivalence bimodule $\F_u^N(B\rtimes_{\delta,\twu}\dualG)$
 that there exists a unique crossed-product norm $\|\cdot\|_\pn$ on $C_c(N, B\rtimes_{\delta,\twu}\dualG)$
 which admits a dual coaction $(\widehat{\dual\delta^\twu})_\pn$ of $N$ such that
 $\F_u^N(B\rtimes_{\delta,\twu}\dualG)$ factors through a $\delta$--$\Inf(\widehat{\dual\delta^\twu})_\pn$ equivariant
 $B$--$(B\rtimes_{\delta,\twu}\dualG)\rtimes_\pn N$ equivalence bimodule. Since all bimodule operations
 extend the operations on the dense submodule $\F_c^N(B\rtimes_{\delta,\twu}\dualG)$, it follows that
 $B\cong B_\pn^N$ with isomorphism given via the identity on $B_c$ (or even on $(B\rtimes_{\delta,\twu}\dualG)_c^N$).
 Thus, the theorem will follow from item (iii) of Theorem~\ref{theo:LandstadTwistedTheorem1}
 if we can show that this isomorphism intertwines the twists $\twu_\pn$ and $\twu$. The latter
 will follow if we can show that the corresponding homomorphisms $\zeta_\twu, \zeta_{\twu_\pn}:C_0(G/N)\to \M(B)$ coincide.
 Recall that we regard $B_c$ as a subset of $B_c^N=(B\rtimes_{\delta,\twu}\dualG)_c^N$ via the identification
 $B_c\cong k_B(B_c)\subseteq B_c^N$.
 By item (i) of Theorem~\ref{theo:LandstadTwistedTheorem1} we have $\zeta_{\twu_\pn}(f)k_B(b)=k_G(f)k_B(b)$ for
 all $f\in C_c(G/N)$, $b\in B_c$. On the other hand, since $(k_B, k_G)$ preserves the twist $\twu$, Remark \ref{rem:PreservationOfTwist}
 implies that $k_B(\zeta(f)b)=k_B(\zeta(f))k_B(b)=k_G(f)k_B(b)$, which shows the desired identity.
\end{proof}

As a direct corollary of the last assertion of the theorem we get:

\begin{corollary}[Stabilization trick for arbitrary coactions]
Every twisted coaction $(\delta,\twu)$ of $(G,G/N)$ is Morita equivalence to an inflated
twisted coaction $(\inf \epsilon, 1)$ for some coaction $\epsilon$ of $N$.
\end{corollary}

\begin{remark}
In \cite{Echterhoff-Raeburn:The_stabilisation_trick} the stabilization trick has been proved for twisted (reduced) coactions of $(G,G/N)$ with $N$ amenable. As remarked in \cite{Kaliszewski-Quigg:Imprimitivity} (see comments before Theorem~5.5 in \cite{Kaliszewski-Quigg:Imprimitivity}), the same ideas carry over to prove a stabilization trick for twisted (full) coactions for arbitrary (non-amenable) $N$ under the assumption that the underlying $G$-coaction is normal. Our result works for all twisted coactions $(\delta,\twu)$.
\end{remark}

Using the above results, we may now generalize the notion of "maximal coactions", "normal coactions", and
"$\pn$-coactions" for a given crossed-product norm $\|\cdot\|_\pn$ on $C_c(G, B\rtimes_{\delta}\dualG)$ as discussed
in \cite{Echterhoff-Kaliszewski-Quigg:Maximal_Coactions, Buss-Echterhoff:Exotic_GFPA, Kaliszewski-Landstad-Quigg:Exotic-coactions} to the category of twisted coactions:

\begin{definition}
Let $(\delta,\twu)$ be a twisted coaction of $(G, G/N)$ on a \cstar{}algebra $B$.
Let $\|\cdot\|_\pn$ be the unique norm on $C_c(N, B\rtimes_{\delta,\twu}\dualG)$
(given by Theorem~\ref{theo:pn-twisted-coactions}) such that $\F_\un^N(B\rtimes_{\delta,\twu}\dualG)$
factors through a $B$--$(B\rtimes_{\delta,\twu}\dualG)\rtimes_\pn N$ Morita equivalence.
We then say that $(\delta,\twu)$ is a  $\pn$-twisted coaction on $B$.
If $\pn=\un$, we say that $(\delta,\twu)$ is a maximal twisted-coaction and if $\pn=\red$, we say that $(\delta,\twu)$
is a normal twisted-coaction.
\end{definition}

\begin{remark}\label{rem-maximalization}
Let $(B,\delta,\twu)$ be a twisted coaction of $(G,G/N)$, and consider the weak $G\rtimes N$-algebra $A=B\rtimes_{\delta,\twu}\dualG$.
Then the twisted coaction $(\delta_u^N,\twu_u)$ on $A^N_\un$ serves as a {\em maximalization} of $(\delta,\twu)$ and $(\delta_\red^N, \twu_\red)$ on $A^N_\red$
serves as a {\em normalization} of $(\delta,\twu)$, while, for an arbitrary crossed-product norm $\|\cdot\|_\pn$ which admits a
dual coaction $(\dual\delta^\twu)_\pn$, $(\delta_\pn^N, \twu_\pn)$ may be regarded as a \emph{$\pn$-ization} of $(\delta,\twu)$.

In this language, $(\delta,\twu)$ is a $\mu$-coaction if and only if $(\delta,\twu)\cong (\delta_\pn^N,\twu_\pn)$.
Thus we see that we get complete twisted analogues of the results obtained in \cite{Buss-Echterhoff:Exotic_GFPA}.
Although we do not develop this here, we remark that it is also possible to obtain an analogue of the categorical
Landstad Duality Theorem  \cite{Buss-Echterhoff:Exotic_GFPA}*{Theorem~7.2} for twisted coactions by using essentially the same ideas as used there.
\end{remark}

The following result follows immediately from item (ii) of Theorem~\ref{theo:LandstadTwistedTheorem1}.

\begin{corollary}
A twisted coaction $(B,\delta,\twu)$ of $(G,G/N)$ is maximal (resp. normal) if and only if the (untwisted) coaction $(B,\delta)$
is a maximal (resp. normal) coaction of $G$. In particular, if $N$ is an amenable closed subgroup of $G$, then
every $G$-coaction $(B,\delta)$ which is twisted over $G/N$ is both maximal and normal.
\end{corollary}

Recall that a \emph{unitary coaction} is a $G$-coaction $(B,\delta)$ which is twisted over $G$ (that is, $N=\{e\}$ is the trivial group in the above notation).
Equivalently, this is the same as a weak $G\rtimes\{e\}$-algebra, that is, a \cstar{}algebra $B$ with a nondegenerate representation $\phi\colon \contz(G)\to \M(B)$.
The $G$-coaction $\delta$ is then recovered by the formula $\delta(b)=(\phi\otimes\id)(\omega_G)(b\otimes 1)(\phi\otimes\id)(\omega_G)^*$.
The above corollary immediately implies the following result (see also \cite{Deicke:Pointwise}*{Proposition~A1}).

\begin{corollary}
Every unitary coaction is maximal and normal.
\end{corollary}

As already mentioned previously, the following decomposition theorem is well-known (it has been proved by Phillips and Raeburn
in \cite{Phillips.Raeburn:Twisted} for amenable $N$ and reduced coactions. But in \cite{Quigg-Raeburn:Induced}*{Remark~7.12}
Quigg and Raeburn stated that the amenability of $N$ is actually not necessary if one works with full coactions).
As an application of our methods, we now derive an alternative proof for this theorem:

\begin{corollary}[Phillips-Raeburn]
\label{cor:PR-DecompositionTheo}
For an arbitrary $G$-coaction $(B,\delta)$ and a normal closed subgroup $N\sbe G$, there is a canonical isomorphism of weak $G\rtimes N$\nb-algebras:
$$B\rtimes_\delta\dualG\cong (B\rtimes_{\delta|}\dual{G/N})\rtimes_{\tilde\delta,\tilde\twu}\dualG,$$
where $(\tilde\delta,\tilde\twu)$ denotes the twisted $(G,G/N)$-coaction on $B\rtimes_{\delta|}\dual{G/N}$ as described  in
Example~\ref{ex:twisted_coactions} above.
\end{corollary}
\begin{proof}
Let $A$ be the weak $G\rtimes N$-algebra $B\rtimes_\delta\dualG$.
For the crossed-product norms $\pn=\un$ or $\pn=\red$, it follows from
Theorem~\ref{thm:GPFA(BxG)N} that
$A_\pn^N\cong B\rtimes_{\delta_\pn|}\dual{G/N}$ and we leave it as an exercise for the reader to check that
the isomorphism  is equivariant for the twisted coaction $(\delta_\pn^N, \twu_\pn)$ and the decomposition coaction
$(\tilde\delta_\pn,\tilde\twu_\pn)$.
It follows from Theorem~\ref{theo:LandstadTwistedTheorem1}(iii) that we have a natural isomorphism
\begin{equation}\label{eq-isotwist}
\big(B_\pn\rtimes_{\delta_\pn|}\dual{G/N}\big)\rtimes_{\tilde\delta_\pn,\tilde\twu_\pn}\dualG\cong A^N_\pn\rtimes_{\delta^N_\pn,\twu^N_\pn}\dualG\cong A
\end{equation}
 of weak $G\rtimes N$-algebras for $\mu=u$ and $\mu=\red$.
On the other hand, since $(B_\un,\delta_\un)$ is the maximalization and $(B_\red,\delta_\red)$ is the normalization of $(B,\delta)$
there are equivariant surjections $B_u\onto B\onto B_\red$
which therefore induce surjections
$$B_\un\rtimes_{\delta_u|}\dual{G/N}\onto B\rtimes_{\delta|}\dual{G/N}\onto B_\red\rtimes_{\delta_\red|}\dual{G/N}$$
 which are morphisms of $(G,G/N)$-coactions and hence also induce surjections
$$\big(B_\un\rtimes_{\delta_u|}\dual{G/N}\big)\rtimes_{\tilde\delta_\un,\tilde\twu_\un}\dualG\onto \big(B\rtimes_{\delta|}\dual{G/N}\big)\rtimes_{\tilde\delta,\tilde\twu}\dualG\onto
\big(B_\red\rtimes_{\delta_\red|}\dual{G/N}\big)\rtimes_{\tilde\delta_\red,\tilde\twu_\red}\dualG.$$
Moreover, by Equation~\eqref{eq-isotwist}, the composition of the two epimorphisms above is an isomorphism and the first and the third algebra are isomorphic to
$B\rtimes_\delta\dualG$, so  $\big(B\rtimes_{\delta|}\dual{G/N}\big)\rtimes_{\tilde\delta,\tilde\twu}\dualG$
must be also isomorphic to $B\rtimes_\delta\dualG$, as desired.
\end{proof}

We finish with the following consequence of  the Landstad Duality Theorem~\ref{theo:LandstadTwistedTheorem1},
which shows that Mansfield's Imprimitivity Theorem~\ref{thm:GPFA(BxG)N} can be enriched to an equivalence of twisted coactions.
This therefore yields a natural connection between the two main topics of this paper.

\begin{corollary}
Let $(B,\delta)$ be a  maximal coaction of $G$.
Then there is a coaction $\delta_{\F_u^N}$ on  Mansfield's
$B\rtimes_{\delta|}\dual{G/N}$--$B\rtimes_\delta\dualG\rtimes_{\dual\delta|}N$
imprimitivity bimodule $\F_u^N(B\rtimes_\delta\dualG)$
which is compatible with the canonical twisted coactions on both algebras, namely,
the decomposition twisted coaction $(\tilde\delta,\tilde\omega)$ on
$B\rtimes_{\delta|}\dual{G/N}$ and the trivially twisted coaction $(\Inf\dual{\,\dual\delta|\,},1)$ on $B\rtimes_\delta\dualG\rtimes_{\dual\delta|}N$.
In other words, $(\F_u^N(B\rtimes_\delta\dualG),\delta_{\F^N_u})$ is a Morita equivalence of twisted coactions
$$(B\rtimes_{\delta|}\dual{G/N},\tilde\delta,\tilde\omega)\sim (B\rtimes_\delta\dualG\rtimes_{\dual\delta|}N,\Inf\widehat{\,\dual\delta|\,},1).$$
A similar result holds for normal coactions $(B,\delta)$ if we replace the universal crossed products by the reduced crossed products everywhere.
\end{corollary}
\begin{proof}
This follows directly from Theorem~\ref{theo:LandstadTwistedTheorem1}(ii) (applied to $A=B\rtimes_\delta\dualG$) and the fact (already observed in the proof Corollary~\ref{cor:PR-DecompositionTheo}) that the decomposition twisted coaction $(\tilde\delta,\tilde\omega)$ corresponds to the twisted coaction $(\delta^N_\pn,\omega_\pn)$ under the canonical isomorphism $B\rtimes_{\delta|}\dual{G/N}\cong (B\rtimes_\delta\dualG)^N_\pn$.
\end{proof}

\begin{remark}
We should remark that for maximal coactions, the equivalence
$$(B\rtimes_{\delta|}\dual{G/N},\tilde\delta)\sim (B\rtimes_\delta\dualG\rtimes_{\dual\delta|}N,\Inf\widehat{\,\dual\delta|\,})$$
is one of the main results of \cite{Kaliszewski-Quigg:Mansfield}, where it has been also proved that $(B,\delta)\mapsto (\F^N_\un(B\rtimes_\delta\dualG),\delta_{\F^N_\un})$ may be
interpreted as an equivalence between the crossed-product functors $(B,\delta)\mapsto (B\rtimes_{\delta|}\dual{G/N},\tilde\delta)$ and $(B,\delta)\mapsto (B\rtimes_\delta\dualG\rtimes_{\dual\delta|}N,\Inf\widehat{\,\dual\delta|\,})$ if we restrict to maximal coactions $(B,\delta)$. Our result shows that the natural twists involved match up,
so that  $\F^N_\un(B\rtimes_\delta\dualG)$ may be  viewed also as an equivalence between the functors $(B,\delta)\mapsto (B\rtimes_{\delta|}\dual{G/N},\tilde\delta,\tilde\twu)$ and $(B,\delta)\mapsto (B\rtimes_\delta\dualG\rtimes_{\dual\delta|}N,\Inf\widehat{\,\dual\delta|\,},1)$. Moreover, it  follows from our Proposition~\ref{prop:FixingN} that $\F^N_\un(B\rtimes_\delta\dualG)$
 carries a $G$-action which is compatible with the natural twisted actions of $(G,N)$ on the left and right coefficient algebras, namely, the inflation $\Inf\dual{\delta|}$ of the dual $G/N$-action on $B\rtimes_{\delta|}\dual{G/N}$ (viewed as a trivially twisted action of $(G,N)$), and the decomposition twisted action $(\tilde{\dual\delta|},\iota_N)$ on $B\rtimes_\delta\dualG\rtimes_{\dual\delta|}N$. Therefore $\F^N_\un(B\rtimes\dualG)$ also provides an equivalence between the functors $(B,\delta)\mapsto (B\rtimes_{\delta|}\dual{G/N},\Inf\dual{\delta|},1)$ and $(B,\delta)\mapsto (B\rtimes_\delta\dualG\rtimes_{\dual\delta|}N,\tilde{\dual\delta|},\iota_N)$.

An analogue of these equivalences follows also for normal coactions $(B,\delta)$ via the bimodule $\F^N_\red(B\rtimes_\delta\dualG)$. This
case has  been  shown before in \cite[Theorem 4.21]{Echterhoff-Kaliszewski-Quigg-Raeburn:Categorical}; see also \cite{Kaliszewski-Quigg-Raeburn:ProperActionsDuality}.
\end{remark}

\end{document}